\documentclass{amsart}

\usepackage{bbm,amsmath,amsthm,amsfonts,amssymb}
\usepackage[mathscr]{eucal}
\usepackage[english]{babel}

\usepackage{graphicx}

\newcommand\eref[1]{(\ref{#1})}

\newtheorem{theorem}{Theorem}[section]
\newtheorem{lemma}[theorem]{Lemma}

\newtheorem{proposition}[theorem]{Proposition}
\newtheorem{definition}[theorem]{Definition}

\newtheorem{remark}[theorem]{Remark}
\newtheorem{remarks}[theorem]{Remarks}

\numberwithin{equation}{section}

\def\bE{{\mathbb E}}
\def\bN{{\mathbb N}}

\def\bC{{\mathbb C}}
\def\bR{{\mathbb R}}
\def\bL{{\mathbb L}}

\def\bT{{\mathbb T}}
\def\bS{{\mathbb S}}

\def\bB{{\mathbb B}}
\def\bP{{\mathbb P}}
\def\R{\mathbb{R}}

\def\cD{\mathcal{D}}

\def\cF{\mathcal{F}}

\def\cH{\mathcal{H}}

\def\cX{\mathcal{X}}

\def\supp{\operatorname{supp}}
\def\diam{\operatorname{diam}}

\def\diver{\operatorname{div}}
\def\card{\operatorname{card}}
\def\ker{\operatorname{ker}}
\def\spann{\operatorname{span}}

\def\ONE{{\mathbbm 1}}

\def\mm{N}
\def\MM{\mathscr{M}}
\def\bb{{\mathfrak b}}
\def\bPhi{\breve{\Phi}}
\def\cd{{c_\diamond}}

\def\diver{\operatorname{div}}
\def\goto{\longrightarrow}

\def\ww{\breve{w}}

\begin{document}

\title[Kernel and wavelet density estimators]
{Kernel and wavelet density estimators\\ on manifolds and more general metric spaces}

\author[G. Cleanthous, A. Georgiadis, G. Kerkyacharian, P. Petrushev, D. Picard]
{G. Cleanthous, A. G. Georgiadis, G. Kerkyacharian, P. Petrushev,\\ and D. Picard}

\address{Department of Mathematics, Statistics and Physics,
Newcastle University, Newcastle Upon Tyne, NE1 7RU, UK}
\email{galatia.cleanthous@newcastle.ac.uk}

\address{Department of Mathematics and Statistics,
University of Cyprus, 1678 Nicosia, Cyprus}
\email{gathana@ucy.ac.cy}

\address{University Paris Diderot-Paris 7, LPMA, Paris, France}
\email{kerk@math.univ-paris-diderot.fr }

\address{Department of Mathematics,
 University of South Carolina}
\email{pencho@math.sc.edu}

\address{University Paris Diderot-Paris 7, LPMA, Paris, France}
\email{picard@math.univ-paris-diderot.fr}



\thanks{We thank the anonymous referees for a very thorough reading of an earlier version
which enabled us  to substantially improve the paper.
\\Corresponding author: Galatia Cleanthous, E-mail: galatia.cleanthous@newcastle.ac.uk}

\begin{abstract}

We consider the problem of estimating the density of observations taking values in
classical or nonclassical spaces such as manifolds and more general metric spaces.
Our setting is quite general but also sufficiently rich in allowing the development of
smooth functional calculus with well localized spectral kernels, Besov regularity spaces,
and wavelet type systems.
Kernel and both linear and nonlinear wavelet density estimators are introduced and studied.
Convergence rates for these estimators are established,
which are analogous to the existing results in the classical setting of real-valued variables.
\end{abstract}

\date{February 4, 2019}

\maketitle

\section{Introduction}\label{Introduction}

A great deal of efforts is nowadays invested in solving statistical problems,
where the data are located in quite complex domains such as  matrix spaces or surfaces (manifolds).
A seminal example in this direction is the case of spherical data.
Developments in this domain have been motivated by a number of important applications.
We only mention here some of the statistical challenges posed by astrophysical data:
denoising of signals,
testing stationarity, rotation invariance or gaussianity of signals,
investigating  the fundamental properties of the cosmic microwave background (CMB),
impainting of the CMB in zones on the sphere obstructed by other radiations,
producing cosmological maps,
exploring clusters of galaxies or point sources,
investigating the true nature of ultra high energy cosmic rays (UHECR).
We refer the reader to the overview by Starck, Murtagh, and Fadili \cite{starck2010sparse}
of the use of various wavelet tools in this domain as well as the work of some of the authors in this direction \cite{BKMP} and \cite{KPHP}.

Dealing with complex data requires the development of more sophisticated tools and statistical methods
than the existing tools and methods.
In particular, these tools should capture the natural topology and geometry of the application domain.

Our contribution will be essentially theoretical, however, our statements will be illustrated by examples issued from  different fields of applications.

Our purpose in this article is to study the {\em density estimation problem},
namely, one observes $X_1,\dots,X_n$ that are i.i.d. random variables defined on a space $\MM$
and the problem is to find a good estimation to the common density function.

This problem has a long history in mathematical statistics especially when the set $\MM$ is
$\bR^d$ or a cube in $\bR^d$ (see e.g. the monograph \cite{Tsybakov} and the references herein).
Here we will consider very general spaces $\MM$ such as Riemannian manifolds or spaces of
matrices or  graphs and  prove that with  some assumptions,
we can build an estimation theory with estimation procedures, regularity sets and upper bounds evaluations
quite parallel to what has been neatly done in $\bR^d$.
In particular we intend to develop kernel methods with upper bounds and oracle properties
as well as wavelet thresholding estimators with adaptative behavior.

If we want to roughly summarize the basic assumptions that will be made in this work,
let us mention that some of them are concerning the basic dimensional structure of the set (doubling conditions),
whereas others are devoted to construct an environment where regularity spaces
can be defined as well as kernels or wavelets can be constructed.

This setting is quite general but at the same time is sufficiently rich in allowing
the development of smooth functional calculus with well localized spectral kernels, Besov regularity spaces,
and wavelet type systems.
Naturally, the classical setting on $\R^d$ and the one on the sphere are contained in this general framework,
but also various other settings are covered. In particular, spaces of matrices, Riemannian manifolds,
convex subsets of (non-compact) Riemannian manifolds are covered.

As will be shown in this general setting, a regularity scale and a general nonparametric density estimation theory
can be developed in full generality just as in the standard case of $[0,1]^d$ or $\R^d$.
This undertaking requires the development of new techniques and methods that break new ground in the density estimation problem.
Our \textit{main contributions} are as follows:

$(a)$ In a general setting described below, we introduce \textit{kernel density estimators} sufficiently concentrated to
 establish oracle inequalities
and
$\bL^p$-error rates of convergence
for probability density functions lying in \textit{Besov  spaces}.

$(b)$ We also develop linear \textit{wavelet density estimators}
and obtain $\bL^p$-error estimates for probability density functions
in general \textit{Besov} smoothness spaces.

$(c)$ We establish $\bL^p$-error estimates
on \textit{nonlinear wavelet density estimators with hard thresholding} in our general geometric setting.
We obtain such estimates for probability density functions in general Besov spaces. 

To put the results from this article in perspective we next compare them with the results in \cite{CaKPi}.
The geometric settings in both articles are comparable and the two papers study adaptive methods.
In \cite{CaKPi} different standard statistical models
(regression, white noise model, density estimation) are considered in a Bayesian framework.
The methods are different (because we do not consider here Bayesian estimators)
and the results are also different
(since, again, we are not interested here in a concentration result of the posterior distribution).
It is noteworthy that the results in the so called {\it dense case} exhibit
the same rates of convergence.
It is also important to observe the wide adaptation properties of the thresholding estimates here
which allow to obtain minimax rates of convergence in the so called {\it sparse case},
which was not possible in \cite{CaKPi}.

The organization of this article is as follows:
In Section~\ref{sec:setting}, we describe our general setting of a doubling measure metric space
in the presence of a self-adjoint operator
whose heat kernel has Gaussian localization and the Markov property. We provide motivation and inspiration for our developments and we present some first examples, both elementary and more involved.
In Section~\ref{sec:background}, we review some basic facts related to our setting such as
smooth functional calculus, the construction of wavelet frames, Besov spaces, and other background. This section can be read quickly by a reader more motivated by the introduction of estimation procedures.
We develop kernel density estimators in Section~\ref{kdemetric} and establish $\bL^p$-error estimates
for probability density functions in general Besov spaces.
We also introduce and study linear wavelet density estimators.
In Section~\ref{Adaptive}, we introduce and study adaptive wavelet threshold density estimators.
We establish $\bL^p$-error estimates for probability density functions in Besov spaces.
Section~\ref{sec:appendix} is an appendix, where we place the proofs of some claims from previous sections.

\smallskip

\noindent
\textit{Notation:}
Throughout $\ONE_E$ will denote the indicator function of the set $E$
and $\|\cdot\|_p:= \|\cdot\|_{\bL^p(\MM,\mu)}$.
We denote by $c, c'$ positive constants that may vary at every occurrence.
Most of these constants will depend on some parameters that may be indicated in parentheses.
We will also denote by $c_0, c_1, \dots$ as well as $c_\star$, $\cd$ constants that will remain unchanged throughout.
The relation $a\sim b$ means that there exists a constant $c>1$ such that $c^{-1}a \le b\le ca$.
We will also use the notation $a\wedge b :=\min\{a, b\}$, $a\vee b :=\max\{a, b\}$
and $C^k(\bR_+)$, $k\in\bN\cup\{\infty\}$, will stand for the set of all functions with continuous derivatives
of order up to $k$ on $\bR_+:=[0,\infty)$.

\section{Setting and motivation}\label{sec:setting}

We assume that $(\MM,\rho,\mu)$ is a metric measure space equipped with a distance $\rho$ and a
positive Radon measure $\mu$.

Let $X_1,\dots,X_n$  be independent identically distributed (i.i.d.) random variables on $\MM$
with common probability having a density function (pdf) $f$ with respect to the measure $\mu$. Our purpose is to estimate the density $f$.
\\
To an estimator $\hat f_n$ of $f$, we associate its risk:
\begin{align*}
R_n(\hat f, f,p)=\bE_f\left(\int_{\MM}|\hat f_n(x)-f(x)|^p\mu(dx)\right)^{\frac1p}
=\bE_f\|\hat f_n-f\|_p, \quad 1\le p<\infty
\end{align*}
as well as its $\bL_\infty$ risk:
\begin{align*}
R_n(\hat f, f,\infty)=\bE_f\left(\mbox{ess sup}_{x\in \MM}|\hat f_n(x)-f(x)|\right)=\bE_f\|\hat f_n-f\|_\infty.
\end{align*}

We will operate in the following  setting.
Most of the material can be found in an extended form in  the papers \cite{CKP, KP}.
Note that, depending on the results we are going to establish, some of the following conditions will be assumed, others will not.

\bigskip

\subsection{Doubling and non-collapsing conditions}

The following conditions are concerning  properties related to 'dimensional' structure of $\MM$.
\\
{\bf C1.} We assume that the metric space $(\MM,\rho,\mu)$  satisfies the so called {\em doubling volume condition}:
\begin{equation}\label{doubl-0}
\mu(B(x,2r) ) \le c_0\mu(B(x,r))
\quad\hbox{for $x \in \MM$ and $r>0$,}
\end{equation}
where $B(x,r):=\{y\in \MM: \rho(x, y)<r\}$ and $c_0>1$ is a constant.
The above implies that there exist constants $c_0'\ge1$ and $d>0$ such that
\begin{equation}\label{doubling}
\mu(B(x,\lambda r) ) \le c_0'\lambda^d \mu(B(x,r))
\quad\hbox{for $x \in \MM$, $r>0$, and $\lambda >1$,}
\end{equation}
The least $d$ such that (\ref{doubling}) holds is the so called {\em homogeneous dimension} of $(\MM,\rho,\mu)$
\\
From now on we will use the notation
$|E|:=\mu(E)$ for $E\subset \MM$.

\smallskip

In developing adaptive density estimators in Section~\ref{Adaptive}
we will additionally assume that $(\MM, \rho, \mu)$ is a compact measure space with $\mu(\MM)<\infty$
satisfying the following condition:

\smallskip

\noindent
{\bf C1A.} {\em Ahlfors regular volume condition:} There exist constants $c_1,c_2>0$ and $d>0$ such that
\begin{equation}\label{doubling-0}
c_1r^{d}\le |B(x,r)| \le c_2 r^{d}
\quad\hbox{for $x \in \MM$ and $0<r\le \diam(\MM)$.}
\end{equation}
Clearly, condition {\bf C1A} implies conditions {\bf C1} and the following condition {\bf C2} as well,
with $d$ from \eqref{doubling-0} being the homogeneous dimension of $(\MM,\rho,\mu)$.

\smallskip
These doubling conditions have been introduced in Harmonic Analysis in the 70's by R. Coifman and G. Weiss \cite{Coifman-weiss}.

It is interesting already to notice that $d$ will indeed play the role of a dimension in the statistical results as well.
Condition {\bf C1A} is obviously true for $\MM= \bR^d$  with $\mu$ the Lebesgue measure.

Also, the doubling condition  is precisely related to the \textit{metric entropy} using the following lemma
whose elementary proof can be found for instance in \cite[Proposition 1]{CaKPi}.
For $\epsilon >0$, we define,  as usual,  the covering number $N(\epsilon, \MM) $ as the smallest number of balls of radius
$\epsilon $ covering $\MM$.

\begin{lemma}
Under the condition {\bf C1A} and if $\MM$ is compact,
there exist constants $c'>0$, $c''>0$ and $\epsilon_0>0$ such that
\begin{equation}
\label{poly-entropie1} \frac{1}{c'}  \Big(\frac 1\epsilon \Big)^d
\leq N(\epsilon, \MM)   \leq  \frac{2^d}{c''} \Big(\frac 1\epsilon \Big)^d,
\end{equation}
for all $ 0<\epsilon\leq \epsilon_0$.
\end{lemma}

\smallskip

\noindent
{\bf C2.} {\em Non-collapsing condition:} There exists a constant $c_3>0$ such that
\begin{equation}\label{non-collapsing}
\inf_{x\in \MM} |B(x,1)|\ge c_3>0.
\end{equation}
This condition is not necessarily very restrictive. For instance, it is satisfied if $\MM$ is compact. It is satisfied for $\bR^d$ if $\mu$ is the Lebesgue measure, but untrue for $\bR$ if $\mu$ is a Gaussian measure.

\subsection{Smooth operator}

Here comes an important assumption which may seem strange to the reader at first glance.
Before entering into the specificity of the set of assumptions described below, let us explain some motivations.

One rather standard method in density estimation is the kernel estimation method,
i.e. considering a family of functions indexed by $\delta>0$:
$K_\delta: \; \MM\times \MM \goto \bR$ an
 associated kernel density estimator is defined by
\begin{equation}\label{kermeth}
\widehat{K}_\delta(x):=\frac{1}{n}\sum\limits_{i=1}^{n} K_\delta(X_i,x), \quad x\in \MM.
\end{equation}

In $\bR^d$, an important family is the family of translation kernels
$K_\delta(x,y)=[\frac1\delta ]^dG( \frac{x-y}\delta)$, where $G$ is a function $\bR^d\goto \bR$.
When $\MM$ is a more involved set such as a manifold or a set of graphs, of matrices,
the simple operations of translation and dilation may not be meaningful.
Hence, even finding a family of kernels to start with might be a difficulty.
It will be shown in Section \ref{kdemetric} that the following assumptions provide quite 'naturally' a family of kernels.

When dealing with a kernel estimation method, it is standard to consider two quantities :
\begin{align*}
b_\delta(f):=\|\bE_f \widehat{K}_\delta -f\|_p,\quad \|\xi_f\|_p:=\|\widehat{K}_\delta-\bE_f \widehat{K}_\delta\|_p.
\end{align*}

The analysis of the second  term (stochastic term)  $\|\xi_f\|_p$, can be reduced via Rosenthal inequalities
to proper bounds on norms of $K_\delta(\cdot,\cdot)$ and $f$ (see the Lemmas \ref{lem:prep}, \ref{lemmaforlinear}),
where in particular the assumptions of the previous subsection are also important).

The analysis of the first term $b_\delta(f)$ is linked to the approximation properties of the family $\bE_f \widehat{K}_\delta$.
One can stop at this level and precisely express the performance of an estimator in terms of $\|\bE_f \widehat{K}_\delta -f\|_p$.
This is the purpose of oracle inequalities (see Theorem \ref{thm:oracle-ineq}).

However, it might seem more convincing if one can relate the rate of approximation of the family
$\|\bE_f \widehat{K}_\delta -f\|_p$ to regularity properties of the function $f$.
It is standardly proved (see e.g. \cite{HKPT}),
that in $\bR^d$ if $K$ is a translation family with mild properties on $K$,
then polynomial rates of approximation are obtained for functions with Besov regularity.

Hence, an important issue becomes finding spaces of regularity associated to a possibly complex set $\MM$.
On a compact metric space $(M, \rho)$ one can always define the scale of $s$-Lipschitz spaces
defined by the following norm
\begin{equation}\label{def-Lip}
\|f\|_{Lip_s} := \|f\|_\infty + \sup_{x\ne y}  \frac{|f(x)-f(y)|}{\rho(x,y)^s},\quad 0<s\le 1.
\end{equation}

In Euclidian spaces  a function  can be much more regular than Lipschitz, for instance differentiable at different orders,
or belong to some Sobolev or Besov spaces.

When $\MM$ is a set where there is no obvious notion of differentiability, one can make the observation that
in $\bR^d$ or Riemannian manifolds, regularity properties can also be expressed via the associated Laplacian.
The Laplacian itself is an operator of order 2, but its square root is of order 1 and can be interpreted as
a substitute for derivation.

We will use this analogy to introduce an operator $L$ playing the role of a Laplacian.
However, conditions are needed to ensure that this analogy makes sense and can lead to
a scale of spaces with suitable properties (which for instance, for small regularities correspond to Lipschitz spaces).
This is why we adopt the setting introduced in \cite{CKP, KP}. This setting is  rich enough  to develop
a Littlewood-Paley theory in almost complete analogy with the classical case on $\R^d$, see \cite{CKP,KP}.
In particular, it allows to develop Besov spaces $B^s_{pq}$ with all sets of indices.
At the same time this framework is sufficiently general to cover a number of interesting cases
as will be shown in what follows.

\bigskip

Our main assumption is that the space $(\MM,\rho,\mu)$
is complemented by an essentially self-adjoint non-negative operator $L$ on $\bL^2(\MM,\mu)$,
mapping real-valued to real-valued functions,
such that the associated semigroup $P_t=e^{-tL}$ consists of integral operators with the
(heat) kernel $p_t(x,y)$ obeying the following conditions:

\smallskip

\noindent
{\bf C3.} {\em Gaussian localization:} There exist constants $c_4, c_5>0$ such that
\begin{equation}\label{Gauss-local}
|p_t(x,y)|
\le \frac{c_4\exp\Big(-\frac{c_5\rho^2(x,y)}t\Big)}{\big(|B(x, \sqrt{t})||B(y, \sqrt{t})|\big)^{1/2}}
\quad\hbox{for} \;\;x,y\in \MM,\,t>0.
\end{equation}

\noindent
{\bf C4.} {\em H\"{o}lder continuity:} There exists a constant $\alpha>0$ such that
\begin{equation}\label{lip}
\big|  p_t(x,y) - p_t(x,y')  \big|
\le c_4\Big(\frac{\rho(y,y')}{\sqrt t}\Big)^\alpha
\frac{\exp\Big(-\frac{c_5\rho^2(x,y)}t\Big)}{\big(|B(x, \sqrt{t})||B(y, \sqrt{t})|\big)^{1/2}}
\end{equation}
for $x, y, y'\in \MM$ and $t>0$, whenever $\rho(y,y')\le \sqrt{t}$.

\smallskip

\noindent
{\bf C5.} {\em Markov property:}
\begin{equation}\label{hol3}
\int_\MM p_t(x,y) d\mu(y)= 1
\quad\hbox{for $x\in \MM$ and $t >0$.}
\end{equation}

Above $c_0, c_1, c_2, c_3, c_4, c_5, d, \alpha>0$ are structural constants.
These technical assumptions express that fact that the Heat kernel associated
with the operator $L$ 'behaves' as the standard Heat kernel of $\bR^d$.

\subsection{Typical examples}\label{typic}

Here we present some examples of setups that are covered by the setting described above.
We will use these examples in what follows to illustrate our theory.
More involved examples will be given Section~\ref{Examples}.

\subsubsection{Classical case on $\MM=\bR^d$}

Here $d\mu$ is the Lebesgue measure and $\rho$ is the Euclidean distance on $\bR^d$.
In this case we consider the operator
$$ -L(f)(x)= \sum_{j=1}^d \partial_i^2 f(x)= {\rm div} (\nabla f)(x)$$
defined on the space $\cD(\bR^d)$ of $C^\infty$ functions with compact support.
As is well known
the operator $L$ is positive essentially self-adjoint and has a unique extension to
a positive self-adjoint operator.
The associate semigroup $e^{t\Delta}$ is given by the operator with the Gaussian kernel:
$p_t(x,y)= (4\pi t)^{-\frac d2} \exp\big(-\frac{|x-y|^2}{4t}\big)$.

\subsubsection{Periodic case on $\MM=[-1,1]$}\label{torus}

Here $d\mu$ is the Lebesgue measure and $\rho$ is the Euclidean distance on the circle.
The operator is $L(f)=-f''$ defined on the set on infinitely differentiable periodic functions.  
It has eigenvalues $k^2\pi^2$ for $k\in \bN_0$ and eigenspaces
$$ 
\ker(L)=\cH_0=  \spann \Big\{ \frac{1}{\sqrt 2}\Big\},
\quad
\ker(L- k^2\pi^2)= \cH_k={\rm span}\,\{\cos k\pi x, \sin k\pi x \}.
$$

\subsubsection{Non-periodic case on $\MM=[-1,1]$ with Jacobi weight}\label{jacobi}
(This example is further developed in Subsection 6.1.)
Note that this example can arise when dealing with data issued from a density which
itself has received a folding treatment such as in the Wicksell problem (\cite{josi90,needvd}).
Now, the measure is 
$$
d\mu_{\alpha,\beta}(x)= W_{\alpha, \beta}(x)dx= (1-x)^\alpha(1+x)^\beta dx,
\quad \alpha,\beta>-1, 
$$ 
the distance $\rho$ is the Euclidean distance, and $L$ is the Jacobi operator 
$$
-L(f) = \frac 1{W_{\alpha,\beta}(x)} D_x \big((1-x^2)W_{\alpha,\beta}(x) D_x\big) f.
$$
Conditions {\bf C1-C5} are satisfied, but not the Ahlfors condition {\bf C1A}, unless $\alpha=\beta=-\frac 12.$
The discrete spectral decomposition of $L$ is given by one dimensional spectral spaces:
$$
\bL^2(M,\mu_{\alpha,\beta})=\bigoplus E_{\lambda_k^{\alpha,\beta} },
\quad 
E_{\lambda_k^{\alpha,\beta}}
= \ker(L-\lambda_k^{\alpha,\beta}I_d) = \spann\big\{ p_k^{\alpha, \beta}(x)\big\}, 
$$ 
where $p_k^{\alpha, \beta}(x) $ is the $k$th degree Jacobi polynomial and
$\lambda_k^{\alpha,\beta}=k(k+\alpha +\beta +1)$.
 
\subsubsection{Riemannian manifold $\MM$ without boundary}\label{rieman} %
If $\MM$ is a Riemannian manifold,
then the Laplace operator $\Delta_M$ is well defined on $M$ (see \cite{grigor}) and we consider
$$
L=-\Delta_M.
$$
If $\MM$ is compact, then conditions {\bf C1-C5} are verified, including the Ahlfors condition {\bf C1A}.
Furthermore, there exists an associated discrete spectral decomposition 
with finite dimensional spectral eigenspaces of $L$:
$$
\bL^2(\MM,\mu)= \bigoplus E_{\lambda_k},
\quad
E_{\lambda_k} = \ker(L-\lambda_kI_d),
\quad  
\lambda_0=0 <\lambda_1 < \lambda_1 <\cdots.
$$

\subsubsection{Unit sphere $\MM=\bS^{d-1}$ in $\bR^d, \; d\geq 3$} \label{sphere} 

This is the most famous  Riemannian manifold with the induced structure from $\bR^d$.
Here $d\mu$ is the Lebesgue measure on $\bS^{d-1}$,
$\rho$ is the geodesic distance on $\bS^{d-1}$:
$$
\rho(\xi, \eta)=  \arccos (\langle\xi, \eta \rangle_{\bR^d}),
$$
and $L:=-\Delta_0$ with $\Delta_0$ being the Laplace-Beltrami operator on $\bS^{d-1}$.
The spectral decomposition of the operator $L$ can be described as follows:
$$
\bL^2(\bS^{d-1},\mu)= \bigoplus E_{\lambda_k},
\quad 
E_{\lambda_k}= ker(L-\lambda_kI_d),
\quad \lambda_k= k(k+d-2).  
$$
Here $E_{\lambda_k}$ is the restriction to $\bS^{d-1}$ of harmonic homogeneous polynomials of degree $k$
(spherical harmonics), see \cite{Stein}.
We have $\dim(E_{\lambda_k})=\binom{d-1}{d+k-1}- \binom{d-1}{d+k-3}.$

\subsubsection{Lie group of matrices: $\MM=SU(2)$}\label{su2}

This example is interesting in astrophysical problems, especially in the measures associated to the CMB, where instead of only measuring the intensity of the radiation we also measure spins.
By definition
$$
SU(2) := \Big\{\left(\begin{array}{cc}a & b \\-\overline b &\overline a\end{array}\right), 
\quad a,b \in \bC, |a|^2 +|b|^2=1 \Big\}.
$$
Thus  
$$
q \in SU(2) \leftrightarrow  q\in  M(2, \bC), 
\quad q^{-1}=-q^*, 
\quad \det (q)=1.
$$
This a compact group which topologically is the sphere $\bS^3 \subset \bR^4.$
So, if
$$ 
x=  \left(\begin{array}{cc} x_1+ix_2 & x_3 +i x_4 \\-(x_3 -i x_4) &x_1- ix_2 \end{array}\right); \; 
y=  \left(\begin{array}{cc} y_1+iy_2 & y_3 +i y_4 \\-(y_3 -i y_4) &y_1- iy_2 \end{array}\right)
$$
with $\|y \|^2=\sum_i y_i^2=1= \|x \|^2=\sum_i x_i^2$, 
then
$$ 
\langle x,y \rangle_4=\sum _i x_iy_i= \frac 12 Tr [xy^*].
$$
Thus
$$
\rho_{SU(2)}(x,y)= \arccos \frac 12 Tr [xy^*])
$$
and for any $q,x,y \in SU(2)$
$$
v=\rho_{SU(2)}(qx,qy)=\rho_{SU(2)}(xq,yq)=\rho_{SU(2)}(x,y).
$$
The eigenvalues of $L= -\Delta$ are $\lambda_k= k(k+2)$ and the dimension of the respective eigenspaces  
$E_{\lambda_k}$ is $(k+1)^2.$

\bigskip

\begin{remark}\label{boundaries}
Looking at some of these examples an important question already arises: how to choose in a given problem the distance $\rho$ as well as the dominating measure  before even choosing the operator $L$ and a class of regularity?
In $\bR^d$, most often the euclidean distance and the Lebesgue measure seems  more or less unavoidable.
In some other cases it might not be   so obvious.

Let us take for instance the simple case of $\MM$ being an interval $[-1,1]$. The cases of the ball, the simplex (see Section~\ref{Examples}) and more generally sets with boundaries give rise in fact to identical discussions therefore we will focus on the case of the interval.

So, if $\MM=[-1,1]$, a possible  choice -and probably the most standard one in statistical examples could be taking $\rho$ as the euclidean distance and $\mu$ as the Lebesgue measure.
 Then  the usual translation kernels are available  as well as the standard wavelet bases. However ``something" -which generally is often swept under the carpet or not really detailed- has to be ``done" about the boundary points $\{-1, 1\}$. Often special regularity conditions are assumed about these boundary points such as $f(-1)=f(1)=0$ (subsection~\ref{torus}),  which de facto lead to different methods for representing the functions to be estimated.

Let us now look at the choices  (again for the interval $[-1,1]$ Subsection~\ref{jacobi}) that are made in the ``Jacobi" case.  The distance $\rho(x,y) =|\arccos x -\arccos y|$, suggests a one-to-one correspondance with the semi-circle. The measure $\mu$ ($d\mu_{\alpha,\beta}(x)  = (1-x)^{\alpha}(1+x)^{\beta} dx, \quad \alpha, \beta>-1$) suggests that  the points in the middle of the interval (say $[-\frac12,+\frac12]$, where the measure behaves as the Lebesgue measure) will not be weighted in the same way as the points near the boundary. And in some cases, this  makes perfect sense: for instance if one needs to give a hard weight on these points  because they require special attention, or at the contrary a small one.

Apart from these considerations, there are in fact two measures in the family $\mu_{\alpha,\beta}$, that are undeniable in the case $\MM=  [-1,1]$ equipped with the distance $\rho(x,y) =|\arccos x -\arccos y|$. The first one is the Lebesgue measure (because Lebesgue is always undeniable), corresponding to $\alpha=\beta=0$.
The second one is $\mu_{-\frac12,-\frac12}$, because in that case there is a one-to-one identification between $(\MM,\rho,\mu_{-\frac12,-\frac12})$ and the semi circle equipped with the euclidean distance and Lebesgue measure.

If we look more precisely into these two choices, we see that for the last case,  all the required conditions including the Ahlfors one are satisfied, and the dimension $d=1$, which is  intuitively expected.
Let us now observe that the case of the Lebesgue measure  $\mu_{0,0}$ would lead to a larger dimension $d=2$.
\end{remark}

\section{Background}\label{sec:background}

In this section we collect some basic technical facts and results related to the setting described in Section~\ref{sec:setting}
that will be needed for the development of density estimators.
Most of them can be found in \cite{CKP, GKKP, KP}.

\subsection{Functional calculus}\label{subsec:func-calc}

A key trait of our setting is that it allows to develop a smooth functional calculus. If we recall that the operator $L$ has been introduced as a substitute for Laplacian, we also have  to recall that in $\bR^d$,  regularity properties of the functions are most often expressed in terms of Fourier transforms which are corresponding to spectral decompositions of the Laplacian. Hence there is no surprise that we will consider the spectral decomposition of $L$ and define an associated functional calculus.

Let $E_\lambda$, $\lambda \ge 0$, be the spectral resolution associated with the operator $L$ in our setting.
As $L$ is non-negative, essentially self-adjoint
and maps real-valued to real-valued functions, then for any real-valued, measurable, and
bounded function $h$ on $\R_+$ 
\begin{equation}
h(L):=\int_0^\infty h(\lambda)dE_\lambda,\label{def:multiplier}
\end{equation}
is well defined on $\bL^2(\MM)$.
The operator $h(L)$, called \textit{spectral multiplier},
is bounded on $\bL^2(\MM)$, self-adjoint, and maps real-valued to real-valued functions \cite{Yosida}.
We will be interested in integral spectral multiplier operators $h(L)$.
If $h(L)(x, y)$ is the kernel of such an operator, it is real-valued and symmetric.
From condition {\bf C4} of our setting we know that $e^{-tL}$ is an integral operator
whose (heat) kernel $p_t(x, y)$ is symmetric and real-valued: $p_t(y,x) = p_t(x, y)\in \R$.

\subsubsection{Examples}\label{typic-sm}

Let us revisit some of the examples given in Subsection~\ref{typic}:

(a) Let $\MM=[-1,1]$ be in the periodic case (Subsection~\ref{torus}).
It is readily seen that the projection operators are:
$$P_{\cH_0}(x,y)= \frac 12,\quad  P_{\cH_k}(x,y)= \cos k\pi(x-y).$$
Hence, formally,
$$
h( L)(x,y)=  \frac 12h(0) + \sum_{ k\ge 1} h( k^2\pi^2) \cos k\pi(x-y),
\quad x,y\in [-1,1].
$$

(b) If $\MM$ is a Riemanian manifold (Subsection~\ref{rieman}), then
$h(L)$ is a kernel operator with kernel
$$
h( L)(x,y)=\sum_k  h({ \lambda_k}) P_{k}( x,y)
$$
with $P_{k}(x,y)= \sum_i v_i^{\lambda_k}(x)\overline{v_i^{\lambda_k}(y)} $,
where $ v_i^{\lambda_k}(x), i=1,\ldots, \dim(E_{\lambda_k})$
is an orthonormal basis of $E_{\lambda_k}.$

\smallskip

(c) In the case of the sphere (Subsection \ref{sphere}),
the orthogonal projector operator
$P_{E_{\lambda_k}}: \bL^2 (\bS^{d-1}) \mapsto E_{\lambda_k}$
is a kernel operator with kernel of the form
$$
P_{E_{\lambda_k}}(\xi,\eta)= L_k(\langle\xi, \eta \rangle_{\bR^d}),
\;\;\hbox{where}\;\;
L_k (x)= \frac 1{|\bS^{d-1}|}\Big(1+\frac{k}{\nu}\Big)C_k^{\nu}(x), \;\; \nu=\frac{d-2}2.
$$
Here $C_k^{\nu}(x)$ is the Gegenbauer polynomials of degree $k$.
Usually, the polynomials $\{C_k^{\nu}(x)\}$ are defined by the generating function
$$
\frac 1{(1-2rx+r^2)^\nu}=\sum_{k\ge 0}r^k C_k^{\nu}(x),
\quad |r|<1, \; |x|<1.
$$
Hence, formally
$$
h( L)(\xi,\eta)=\sum_k  h(k(k+d-2))L_k(\langle\xi, \eta \rangle_{\bR^d}),
\quad \xi,\eta\in\; \bS^{d-1}.
$$

\smallskip

(d) In the case of $SU(2)$ (Subsection~\ref{su2}),
the orthogonal projector operator
$P_{E_{\lambda_k}}: \bL^2 (SU(2)) \mapsto E_{\lambda_k}$
is the operator with kernel
$$
P_{E_{\lambda_k}}(f)(\xi, \eta)= L_k\big(\frac 12 Tr [\xi \eta^*]\big)
$$
where $L_k (x)= \frac 1{|\bS^3|}(1+k)C_k^{1}(x)$.
Hence, formally
$$
h( L)(\xi,\eta)=\sum_k h(k(k+2)) L_k\big(\frac 12 Tr [\xi \eta^*]\big),
\quad \xi,\eta\in\; SU(2).
$$

Our further development will heavily depend on the following result from the
smooth functional calculus induced by the heat kernel, developed in \cite[Theorem~3.4]{KP}.
It asserts the localization properties of general spectral multipliers of
the form $g(\delta \sqrt L)$ (corresponding to  functions of the form $h(u)=g(\delta\sqrt{u})$ in \eref{def:multiplier}).
Again the appearance of the square root is by analogy with the Laplacian,
which is an operator of degree 2. It is also interesting to remark that (\ref{local-g})
is valid in $\bR^d$ when $g(\delta \sqrt L)$ is replaced by  $[\frac1\delta ]^dG( \frac{x-y}\delta)$,
where $G$ is a bounded compactly supported function $\bR^d\goto \bR$ for instance.
This result is a building block for the properties of the kernel estimators defined in the sequel.

\begin{theorem}\label{thm:S-local-kernels}
Let $g\in C^\mm(\bR)$, $\mm>d$, be even, real-valued, and $\supp g\subset [-R,R]$, $R>0$.
Then $g(\delta \sqrt L)$, $\delta>0$, is an integral operator with kernel $g(\delta \sqrt L)(x, y)$
satisfying
\begin{equation}\label{local-g}
\big|g(\delta \sqrt L)(x, y)\big|
\le c |B(x, \delta)|^{-1}\big(1+\delta^{-1}\rho(x, y)\big)^{-\mm+\frac{d}{2}},\;\forall\; x,\; y \; \in \; \MM,
\end{equation}
where
$c>0$ is a constant depending on $\|g\|_{\infty}$, $\|g^{(N)}\|_{\infty}$, $\mm$, $R$
and the constants $c_0, c_4, c_5$ from our setting.

Furthermore, for any $\delta>0$ and $x\in \MM$
\begin{equation}\label{int-g}
\int_\MM g(\delta \sqrt L)(x, y)d\mu(y)=g(0).
\end{equation}
\end{theorem}

\subsection{Geometric properties}

Conditions {\bf C1} and {\bf C2} yield
\begin{equation}\label{DNC}
|B(x, r)|\ge (c_3/c_0) r^d, \quad x\in \MM,\; 0<r\le 1.
\end{equation}

To compare the volumes of balls with different centers $x, y\in \MM$ and the same radius $r$
we will use the inequality
\begin{equation}\label{D2}
|B(x, r)| \le c_0\Big(1+ \frac{\rho(x,y)}{r}\Big)^d  |B(y, r)|,
\quad x, y\in \MM, \; r>0.
\end{equation}
As $B(x,r) \subset B\big(y, \rho(y,x) +r\big)$ the above inequality is immediate from (\ref{doubling}).

\smallskip

We will also need the following simple inequality (see \cite[Lemma 2.3]{CKP}):
If $\tau >d$, then for any $\delta>0$
\begin{equation}\label{tech-1}
\int_{\MM} \big(1+\delta^{-1}\rho(x, y)\big)^{-\tau} d\mu(y) \le c|B(x, \delta)|, \quad x\in \MM,
\end{equation}
where $c=(2^{-d}-2^{-\tau})^{-1}$.

\subsection{Spectral spaces}\label{subsec:spectral-sp}

We recall the definition of the spectral spaces $\Sigma_\lambda^p$, $1\le p\le \infty$, from \cite{CKP}.
Denote by $C^\infty_0(\R)$ the set of all even real-valued compactly supported functions.
We define
\begin{equation*} 
\Sigma_\lambda^p
:=\big\{f\in \bL^p(\MM): \theta(\sqrt{L})f=f \; \hbox{for all}\; \theta\in C^\infty_0(\R),\; \theta\equiv 1 \;\hbox{on}\; [0, \lambda] \big\},
\;\; \lambda>0.
\end{equation*}

We will need the following proposition (Nikolski type inequality):

\begin{proposition}\label{prop:nikolski}
Let $1\le p\le q\le\infty$. If $g\in \Sigma_\lambda^p$, $\lambda\ge 1$, then $g\in \Sigma_\lambda^q$ and
\begin{equation}\label{nikolski}
\|g\|_q \le c_\star\lambda^{d(1/p-1/q)}\|g\|_p,
\end{equation}
where the constant $c_\star>1$ is independent of $p$ and $q$.
\end{proposition}

This proposition was established in \cite[Proposition 3.12]{CKP} (see also \cite[Proposition 3.11]{KP}).
We present its proof in the appendix because we need to control the constant $c_\star$.

\subsection{Wavelets}\label{Frames}

In the setting of this article, wavelet type frames for Besov and Triebel-Lizorkin spaces are developed in \cite{KP}.
Here, we review the construction of the frames from \cite{KP} and their basic properties.
Indeed, in this setting the 'wavelets' do not form an orthonormal basis but a frame.
In this case, the construction of a 'dual wavelet system' is necessary to get a representation of type \eref{frame2}.

\smallskip

This construction is inspired by to the Littlewood-Paley  construction of the standard wavelets
introduced by \cite{F-J1},\cite{F-J2}, \cite{F-J3}.

The construction of frames involves a ``dilation" constant $b>1$ whose role is played by $2$ in the wavelet theory on $\R$.

The construction starts with the selection of a function $\Psi_0\in C^{\infty}(\bR_+)$
with the properties:
$\Psi_0(\lambda)=1$ for $\lambda\in[0,1]$,
$0\le\Psi_0(\lambda)\le1$,
and $\supp\Psi_0\subset[0,b]$.
Denote $\Psi(\lambda):=\Psi_0(\lambda)-\Psi_0(b\lambda)$
and set $\Psi_{j}(\lambda):=\Psi(b^{-j}\lambda)$, $j\in\bN$.
From this it readily follows that
\begin{equation}\label{PsiJ}
\sum\limits_{j=0}^{J}\Psi_j(\lambda)=\Psi_{0}(b^{-J}\lambda),
\quad\lambda\in\bR_+.
\end{equation}
For $j\ge 0$ we let $\cX_j \subset \MM$ be a maximal $\delta_j-$net on $\MM$
with $\delta_j:=c_6 b^{-j}$.
It is easy to see that for any $j\ge 0$ there exists a disjoint partition
$\{A_{j\xi}\}_{\xi\in\cX_j}$ of $\MM$
consisting of measurable sets such that
$$
B(\xi, \delta_j/2) \subset  A_{j\xi} \subset B(\xi,\delta_j),
\quad \xi \in \cX_j.
$$
Here $c_6>0$ is a sufficiently small constant (see \cite{KP}).

\begin{lemma}\label{lem:card-Xj}
If $\MM$ is compact, then there exists a constant $c_7>0$ such that
\begin{equation}\label{cardXj}
\card (\cX_j\big)\le c_7 b^{jd},
\quad j\ge0.
\end{equation}
\end{lemma}
\begin{proof}
Assume $\MM$ is compact and
let $\cX_\delta$ be a maximal $\delta$-net on $\MM$, $\delta>0$.
Then
$$
\sum_{\xi\in \cX_\delta} \int_{B(\xi,\delta/2)} \frac{d\mu(x)}{\mu(B(x,\delta)}
\le \int_\MM \frac{d\mu(x)}{\mu(B(x,\delta))}
\le \sum_{\xi\in \cX_\delta} \int_{B(\xi,\delta)} \frac{d\mu(x)}{\mu(B(x,\delta))}.
$$
Therefore, using \eref{doubling} we get
\begin{equation}\label{card-Xdel}
\frac 1{c_0 4^d} \card(\cX_\delta)
\le \int_M\frac{d\mu(x)}{\mu(B(x,\delta)}
\le c_0 2^d \card(\cX_\delta).
\end{equation}
Since $\MM$ is compact we have $\mu(\MM)<\infty$ and
$B(x,D)=\MM$ for $x\in \MM$, where $D$ is the diameter of $\MM$, which is finite (see \cite{CKP}).
Using again \eref{doubling} we get
$\mu(M)= \mu(B(x,D))\le c_0(\frac D{\delta})^d \mu(B(x,\delta)$  for $x \in M$.
Hence
$$
\frac 1{c_0 4^d} \card(\cX_\delta)
\le \int_M\frac{d\mu(x)}{\mu(B(x,\delta))}
\le c_0(D/\delta)^d,
$$
which implies \eqref{cardXj}.
\end{proof}

The $j$th level frame elements $\psi_{j\xi}$ are defined by
\begin{equation}\label{def-frame}
\psi_{j\xi} (x):= |A_{j\xi} |^{1/2} \Psi_j(\sqrt L)(x,\xi),
\quad \xi \in \cX_j.
\end{equation}
We will also use the more compact notation
$\psi_{\xi}:=\psi_{j\xi}$ for $\xi\in\cX_j$.

Let
$\cX:=\cup_{j\ge 0}\cX_j$,
where equal points from different sets $\cX_j$ will be regarded as distinct elements of $\cX$,
so $\cX$ can be used as an index set.
Then $\{\psi_{\xi}\}_{\xi\in\cX}$ is Frame~$\# 1$.

The construction of a dual frame
$\{\tilde\psi_{\xi}\}_{\xi\in \cX}=\cup_j\{\tilde\psi_{j\xi}\}_{\xi\in\cX_j}$
is much more involved;
we refer the reader to \S 4.3 in \cite{KP} for the details.

By construction, the two frames satisfy
\begin{equation}\label{frame1}
\Psi_j(\sqrt{L})(x,y)=\sum\limits_{\xi\in\cX_j}\psi_{j\xi}(y)\tilde\psi_{j\xi}(x),
\quad j\ge0.
\end{equation}

A basic result from \cite{KP} asserts that
for any $f\in \bL^p(\MM, d\mu)$, $1\le p<\infty$,
\begin{equation}\label{frame-decom}
f=\sum_{j\ge 0}\Psi_j(\sqrt{L})f
\quad\hbox{(convergence in $\bL^p$)}
\end{equation}
and the same holds in $\bL^\infty$ if $f$ is uniformly continuous and bounded (UCB) on $\MM$.
As~a~consequence, for any $f\in \bL^p(\MM, d\mu)$, $1\le p\le\infty$, ($\bL^\infty={\rm UCB}$)
we have
\begin{equation}\label{frame2}
f=\sum\limits_{j=0}^{\infty}\sum\limits_{\xi\in\cX_j}\langle f,\tilde\psi_{j\xi}\rangle\psi_{j\xi}
\quad\hbox{(convergence in $\bL^p$).}
\end{equation}
Furthermore, frame decomposition results are established in \cite{KP} for Besov and Triebel-Lizorkin spaces
with full range of indices.

\medskip

\noindent
{\bf Properties of frames in the Ahlfors regularity case.}
We next present some properties of the frame elements in the case when condition {\bf C1A} is stipulated (see \cite{KP}).

\smallskip

1. {\em Localization}: For every $k\in\mathbb{N}$, there exists a constant $c(k)>0$ such that
\begin{equation}\label{bg1}
|\psi_{j\xi}(x)|,\;|\tilde\psi_{j\xi}(x)|\le c(k) b^{jd/2} \big(1+b^{j}\rho(x,\xi)\big)^{-k},
\quad x\in \MM.
\end{equation}

\smallskip

2. {\em Norm estimation}: For $1\le p\le \infty$
\begin{equation}\label{bg2}
\cd^{-1} b^{jd(\frac{1}{2}-\frac{1}{p})}
\le \|\psi_{j\xi}\|_p,\;\|\tilde\psi_{j\xi}\|_p
\le \cd b^{jd(\frac{1}{2}-\frac{1}{p})},
\quad \xi\in \cX_j,\;j\ge 0.
\end{equation}

\smallskip

3. For $1\le p\le \infty$
\begin{equation}\label{bg3}
\Big\|\sum_{\xi\in\cX_j} \lambda_{\xi}\psi_{j\xi}\Big\|_p\le \cd b^{jd(\frac{1}{2}-\frac{1}{p})}
\Big(\sum_{\xi\in\cX_j} |\lambda_{\xi}|^p\Big)^{1/p},
\quad j\ge 0,
\end{equation}
with the usual modification when $p=\infty$.
Above the constant $\cd>1$ depends only on $p$, $b$, $\Psi_0$, and the structural constants of the setting.

\subsection{Besov spaces}\label{subsec:Besov}

We will deal with probability density functions (pdf's)
in Besov spaces associated to the operator $L$ in our setting.
These spaces are developed in \cite{CKP,KP}.
Definition \ref{def:Besov} coincides  in $\bR^d$ with one the definitions of  usual Besov spaces  with $L$ replaced by  Laplacian ($-\Delta$ in fact to get a positive operator).

Here we present some basic facts about Besov spaces that will be needed later on.

Let $\Phi_0,\Phi\in C^{\infty}(\R_+)$ be real-valued functions satisfying the conditions:
\begin{equation}\label{Phi0}
\supp \Phi_0\subset [0,b],\;\Phi_0(\lambda)=1\;\text{for}\;\lambda\in[0,1],\;
\Phi_0(\lambda)\ge c>0\;\text{for}\;\lambda\in[0,b^{3/4}],
\end{equation}
\begin{equation}\label{Phi}
\supp \Phi\subset [b^{-1},b],\;
\Phi(\lambda)\geq c>0\;\text{for}\;\lambda\in[b^{-3/4},b^{3/4}].
\end{equation}
Set $\Phi_j(\lambda):=\Phi(b^{-j}\lambda),\;\text{for}\;j\geq 1$.

\begin{definition}\label{def:Besov}
Let $s>0$, $1\le p \le \infty$, and $0< q\le \infty$.
The \textit{Besov space} $B^s_{pq}=B^s_{pq}(\MM,L)$ is defined as the set of all functions $f\in \bL^p(\MM,\mu)$ such that
\begin{equation}\label{Besovdef}
\|f\|_{B^s_{pq}}:=\Big(\sum_{j\ge 0} \big(b^{sj}\|\Phi_j(\sqrt{L})f\|_p\big)^q\Big)^{1/q}<\infty,
\end{equation}
where the $\ell^q$-norm is replaced by the sup-norm if $q=\infty$.
\end{definition}
Note that as shown in \cite{KP} the above definition of the Besov spaces $B^s_{pq}$ is independent of
the particular choice of $\Phi_0,\Phi$ satisfying \eqref{Phi0}-\eqref{Phi}.
For example with $\Psi_j$ from the definition of the frame elements in \S\ref{Frames}
we have
\begin{equation}\label{Besov-frame}
\|f\|_{B^s_{pq}}\sim\Big(\sum_{j\ge 0} \big(b^{sj}\|\Psi_j(\sqrt{L})f\|_p\big)^q\Big)^{1/q}
\end{equation}
with the usual modification when $q=\infty$.
The following useful inequality follows readily from above
\begin{equation}\label{useful-Besov}
\|\Psi_j(\sqrt{L})f\|_p \le c b^{-sj}\|f\|_{B^s_{pq}},
\quad f\in B^s_{pq}, \; j\ge 0.
\end{equation}


As in $\R^d$, we will need some  embedding results  involving Besov spaces.
Recall the definition of embeddings:
Let $X$ and $Y$ be two (quasi-)normed spaces.
We say that $X$ is continuously embedded in $Y$ and write $X\hookrightarrow Y$
if $X\subset Y$ and for each $f\in X$ we have $\|f\|_Y\le c\|f\|_X$,
where $c>0$ is a constant independent of $f$.

\begin{proposition}\label{Pr:Embed}
$(i)$
If $1\le q\le r\le\infty$, $0<\tau\le\infty$, $s>0$, and $\mu(\MM)<\infty$,
then $B^s_{r\tau}\hookrightarrow B^s_{q\tau}$.

\smallskip

$(ii)$
If $1\le r\le q\le\infty$, $0<\tau\le\infty$ and $s>d\big(\frac{1}{r}-\frac{1}{q}\big)$,
then $B^s_{r\tau}\hookrightarrow B^{s-d(\frac{1}{r}-\frac{1}{q})}_{q\tau}$.

\smallskip

$(iii)$
If $1\le r\le\infty$, $0<\tau\le\infty$, and $s>d/r$, then $B^s_{r\tau}\hookrightarrow \bL^{\infty}$.

\smallskip

$(iv)$
If $1\le p \le r\le\infty$, $0<\tau\le\infty$, $s>0$, and $\mu(\MM)<\infty$, then $B^s_{r\tau}\hookrightarrow \bL^{p}$.
\end{proposition}

To streamline our presentation we defer the proof of this proposition to the appendix.

\medskip

\noindent
{\bf Besov spaces in the Ahlfors regularity case.}
For the development of adaptive density estimators in Section~\ref{Adaptive}
we will need some additional facts from the theory of Besov spaces when condition {\bf C1A} is assumed.
We first introduce the  Besov bodies. 

\begin{definition}
Assume $s>0$, $1\le p\le \infty$, $0< q\le \infty$,
and let $\cX:=\cup_{j\ge 0}\cX_j$ be from the definition of the frames in \S\ref{Frames}.
The  Besov body $\bb^s_{pq}=\bb^s_{pq}(\cX)$ is defined as the set of all
sequences $\{a_\xi\}_{\xi\in\cX}$ of real $($or complex$)$ numbers such that
\begin{equation}\label{Besovdiscrete}
\|a\|_{\bb^s_{pq}}
:=\Big(\sum_{j\ge 0}b^{jsq}
\Big(\sum_{\xi\in\cX_j}\big[b^{-jd(\frac{1}{p}-\frac{1}{2})}|a_{\xi}|\big]^{p}\Big)^{q/p}\Big)^{1/q}<\infty,
\end{equation}
where the $\ell^q$-norm is replaced by the sup-norm if $q=\infty$.
\end{definition}

One of the principle results in \cite{KP} asserts that the Besov spaces $B^s_{pq}$ can be completely characterized
in terms of the  Besov bodies $\bb^s_{pq}$ of the frame coefficients of the respective functions.
To be specific, denote
\begin{equation}\label{bjxi}
\beta_{j\xi}(f):=\langle f,\tilde\psi_{j\xi}\rangle,
\quad \xi\in\cX_j, \; j\geq 0.
\end{equation}
We will also use the more compact notation:
$\beta_{\xi}(f):=\beta_{j\xi}(f)$ for $\xi\in\cX_j$.
In the current setting, assume $s>0$, $1\le p\le \infty$, and $0< q\le \infty$.
In light of \cite[Theorem~6.10]{KP}
$f\in B^s_{pq}$ if and only if $\{\beta_{\xi}(f)\}\in \bb^s_{pq}$
with equivalent norms:
\begin{equation}\label{eqiv-B-b}
\|f\|_{B^s_{pq}}\sim\|\{\beta_{\xi}(f)\}\|_{\bb^s_{pq}}.
\end{equation}
This implies that
if $f\in B^s_{pq}$ for some $s>0$, $p\ge 1$, and $0<q\le \infty$, then
\begin{equation}\label{bg4}
\Big(\sum_{\xi\in\cX_j}|\beta_{j\xi}(f)|^p\Big)^{1/p}
\le c b^{-j(s+d(\frac{1}{2}-\frac{1}{p}))}\|f\|_{B^s_{pq}},
\quad j\ge 0,
\end{equation}
where $c=c(s,p,q)>0$.

\smallskip

By \eqref{bg3} and \eqref{bg4} it follows that, if $f\in B^s_{pq}$ for some $s>0$, $p\ge 1$, and $0<q\le \infty$, then
\begin{equation}\label{bg5}
\Big\|\sum_{\xi\in\cX_j} \beta_{j\xi}(f)\psi_{j\xi}\Big\|_p\le c b^{-sj}\|f\|_{B^s_{pq}},
\quad j\ge 0.
\end{equation}

\section{Kernel density estimators on the metric measure space $\MM$}\label{kdemetric}

Our goal in this section is to introduce and study kernel density estimators (kde's) on a metric measure space $(\MM, \rho, \mu)$
in the general setting described in Section~\ref{sec:setting}.
More precisely, in this section, we assume that conditions {\bf C1}--{\bf C5} are satisfied,
and do not necessarily  assume the Ahlfors regular volume condition {\bf C1A}.

To explain our construction of kernel estimators we begin by considering
the classical example of the periodic case on $\MM=[-1,1]$, presented in Subsection~\ref{torus}.
It will be commonly admitted among nonparametric statisticians to define an estimator of the form
$$
\hat f_T(x)=\frac 12 +\frac1n \sum_{i=1}^n \sum_{1\leq k\le T}  \cos k\pi(x-X_i).
$$
It falls into the category of orthogonal series estimator.
It is well known that these estimators have nice $\bL^2$ properties but can drastically fail
in $\bL^p$, $p\ne 2$, or locally.

In our setting we will replace $\hat f_T$ by a 'smoothed version':
\begin{align}
\hat K_{\delta}(x)
=\frac 12 K(0) +\frac1n \sum_{i=1}^n \sum_{k\ge 1} K(\delta k) \cos k\pi(x-X_i)
=:\frac1n\sum_{i=1}^n K_\delta(x,X_i),\label{kern-torus}
\end{align}
where $K$ is a smooth function $\bR_+\goto\bR$ eventually vanishing at infinity, and
\\
 $K_\delta(x,y)=\frac 12 K(0) + \sum_{k\ge 1} K(\delta k) \cos k\pi(x-y)
$.

\smallskip

In analogy to this case, replacing the circle by $\MM$ and the Laplacian by the operator $-L$
we can naturally introduce kde's on $\MM$ by means of  the machinery of spectral multipliers.

\smallskip

Let $K:\R_+\rightarrow\R$ be a bounded and measurable function.
As we already alluded to in \S \ref{subsec:func-calc},
if $E_\lambda$, $\lambda\ge 0$, is the spectral resolution associated with the operator $L$,
then the operator
\begin{equation}\label{Sm}
K(\sqrt{L}):=\int_{0}^{\infty}K(\sqrt{\lambda})dE_{\lambda}
\end{equation}
is well defined on $\bL^2(\MM)$.
Furthermore, the operator $K(\sqrt{L})$ is self-adjoint and bounded on $\bL^2(\MM)$
with norm
$
\|K(\sqrt{L})\|_{2\rightarrow2} :=\sup_{h,\|h\|_2\le 1}\|H(h)\|_{2}
\le \|K\|_{\infty}.
$

We are interested in multiplier operators $K(\sqrt{L})$ that are integral operators.
In~this case, since the function $K$ is real-valued its \textit{kernel} $K(\sqrt{L})(x,y)$ is real-valued and symmetric.
We will use Theorem \ref{thm:S-local-kernels} to define a family of multiplier operators
whose kernels are suitable for the construction of kernel density estimators on $\MM$.

We now introduce the kde's in the general setting of this article.

\begin{definition}\label{D:kdeM}
Let $X_1,\dots,X_n$ be i.i.d. random variables on $\MM$
in our setting.
Let $K(\delta\sqrt{L})(x, y)$ with $0<\delta\le1$ $($the bandwidth$)$
be the kernel of the integral operator $K(\delta\sqrt{L})$,
where $K:\R_+\rightarrow \R$.
The associated kernel density estimator is defined by
\begin{equation}\label{kde2}
\widehat{K}_\delta(x):=\frac{1}{n}\sum\limits_{i=1}^{n} K(\delta\sqrt{L})(X_i,x), \quad x\in \MM.
\end{equation}
\end{definition}

\begin{remarks}
\noindent
- The analogy with the Torus again explains why we take $\sqrt{L}$ instead of $L$: in the torus, case the eigenvalues are $\pi^2k^2$. It is mostly a comfort choice. Replacing $\sqrt{L}$ by $L$ is possible but would lead to a different regularity scale.
\\
- Again the analogy with the Torus case could lead to a choice of the form $K$ to be the indicator function of  the interval $[0,1]$  for instance. This choice would induce $\bL^2$ properties, but not  $\bL^p$ because this function is not smooth enough to get   the concentration inequalities of Theorem \ref{thm:S-local-kernels}.
\\
- If $K(\lambda)=e^{-\lambda^2}$,
then $K(\delta\sqrt{L})(x,y)=p_{\delta^2}(x,y)$ $($the ``heat kernel"$)$
can be used to define a kernel density estimator. This choice relies to  the Bayesian estimator provided in \cite{CaKPi}.
\\
- This type of kernel estimators although constructed upon orthogonal projectors, because of the smoothing function $K$ will finally have properties which are comparable to translation kernel estimators in $\bR^d$. In $\bR^d$ some properties such as a number of moment annulation (see for instance \cite{Tsybakov}) to get a correct biais are required which will be here replaced by the vanishing properties at infinity of the function $K$ and the smoothness of the function $K$.
\end{remarks}

\subsection{Upper bound estimates for kernel density estimators}

We will especially study kernel density estimators induced by compactly supported
$C^{\infty}$ multipliers, often called Littlewood-Paley functions.
In fact other type of kernels among the family of multipliers could lead to quite similar results. The Littlewood-Paley are especially powerful technically to obtain upper-bounds.
More explicitly,
let $\Phi$ be an even $C^\infty(\bR)$ real-valued function with the following properties:
\begin{equation}\label{prop-Phi}
\supp \Phi \subset [-1,1]
\quad\hbox{and}\quad
\Phi(\lambda)= 1 \;\;\hbox{for}\;\; \lambda\in [-1/2, 1/2].
\end{equation}
By Theorem~\ref{thm:S-local-kernels} it follows that $\Phi(\delta\sqrt{L})$
is an integral operator with well localized symmetric kernel $\Phi(\delta\sqrt{L})(x, y)$
and the Markov property:
\begin{equation}\label{Markov}
\int_{\MM}\Phi(\delta\sqrt{L})(x, y) d\mu(y)=1.
\end{equation}

As before we assume that
$X_1,\dots,X_n$ $(n\ge 2)$ are i.i.d. random variables with values on~$\MM$
and common probability density function (pdf) $f$ with respect to the measure $\mu$ on the space $\MM$.
Let $X_i\sim X$.
We will denote by $\bE=\bE_f$ the expectation with respect to
the probability measure $\bP=\bP_f$. 
We are interested in the kernel density estimator
\begin{equation}\label{def-Phi-kde}
\widehat{\Phi}_\delta(x)=\widehat{\Phi}_\delta(x, X_1,\dots,X_n)
:= \frac 1n \sum_{i=1}^n \Phi(\delta \sqrt L)(x,X_i),\;\forall\; x\; \in \MM.
\end{equation}

\subsection{Examples of Littlewood Paley kernel estimates}
Let us take some examples issued from section~\ref{typic-sm}. We already discussed the case $[-1,1]$ as a Torus in the introduction this section.
\\
- For $[-1,1]$ in the Jacobi framework, (section~\ref{jacobi}), we get the following estimator:
$$\widehat{\Phi}_\delta(x)=\frac 1n \sum_{i=1}^n\sum_{k}\Phi(\delta\sqrt{k(k+\alpha +\beta +1)})p_k^{\alpha, \beta}(x) p_k^{\alpha, \beta}(X_i),\; \forall\; x\in \;[-1,1]$$
where
 $ p_k^{\alpha, \beta}(x) $ is the normalized Jacobi polynomial
( $\int_{-1}^1 |p_k^{\alpha, \beta}(x) |^2 (1-x)^\alpha(1+x)^\beta dx =1$).
\\
- For the sphere, (section~\ref{sphere}), we get the following estimator
$$\widehat{\Phi}_\delta(x)=\frac 1n \sum_{i=1}^n\sum_{k}\Phi(\delta\sqrt{k(k+d-2)})L_k(\langle x,X_i\rangle_{\R^d}),\; \forall\; x\in \;\bS^{d-1}$$
where $$L_k (x)= \frac 1{|\bS^{d-1}|}\big(1+\frac{k}{\nu}\big)C_k^{\nu}(x), \; \nu=\frac{d-2}2$$
\\
- For $SU(2)$, (section~\ref{su2}), we get the following estimator,
$$\widehat{\Phi}_\delta(x)=\frac 1n \sum_{i=1}^n\sum_{k}\Phi(\delta\sqrt{k(k+2)})L_k\big(\frac 12 Tr [X_i x^*]\big),\; \forall\; x\in \;SU(2)$$
with $L_k (x)= \frac 1{|\bS^3|}(1+k)C_k^{1}(x)  $.

\subsection{Upper bound results}
We next study the approximation of pdf's $f$
by such kernel estimators.
We first establish oracle inequalities:

\begin{theorem}\label{thm:oracle-ineq}
Assume $1\le p\le \infty$
and let $\Phi$ be a~Littlewood-Paley function as above.
In the setting described above and with $\widehat{\Phi}_\delta$ from \eqref{def-Phi-kde} we have:

$(i)$ If $2 \le p <\infty$, then
\begin{equation*}
\bE \| \widehat{\Phi}_\delta-f \|_p
\le \frac{c(p)}{(n\delta^d)^{1-\frac 1p}} + \frac{c(p)}{( n\delta^d)^{\frac 12}}\| f\|_{\frac p2}^{\frac 12}
+ \|\Phi(\delta\sqrt{L})f-f \|_p,
\quad 0<\delta\le 1.
\end{equation*}

$(ii)$
If $1\le p <2$ and $\supp(f) \subset B(x_0,R)$ for some $x_0\in \MM$ and $R>0$, then
$$
\bE\|\widehat{\Phi}_\delta-f\|_p
\le \frac{c(p)}{( n\delta^d)^{\frac 12}}|B(x_0,R)|^{\frac 1p-\frac 12}
+ \|\Phi(\delta\sqrt{L})f-f \|_p,
\quad 0<\delta\le 1.
$$

$(iii)$
There exists a constant $c$ such that for any $q\ge 2$ and $0<\delta\le 1$ we have
$$
\bE\|\widehat{\Phi}_\delta-f\|_\infty
\le c\delta^{-\frac dq} \Big(\frac{q}{(n\delta^d)^{1-\frac 1q}}
+ \frac{q^{1/2}}{(n\delta^d)^{\frac 12} }\|f\|_{\infty}^{\frac 12-\frac 1q}\Big)
+ \|\Phi(\delta\sqrt{L})f-f \|_\infty.
$$

\end{theorem}

We next estimate the rates of $\bL^p$-approximation of pdf's $f$ lying in Besov space balls
by kernel estimators.
Denote
\begin{equation}\label{def-Besov-ball-1}
B^s_{p\tau}(m):=\big\{f\;\text{is pdf}: \|f\|_{B^s_{p\tau}}\le m\big\}
\end{equation}
and
\begin{equation}\label{def-Besov-ball-2}
B^s_{p\tau}(m,x_0,R):=\big\{f\in B^s_{p\tau}(m): \supp f\subset B(x_0,R)\big\},
\quad x_0\in \MM, \; m,R>0.
\end{equation}

Here is our main result on the properties of these  estimators, for density functions in Besov spaces, when the risk and the regularity classes are defined with the same norm.

\begin{theorem}\label{thm:upper-bound}
Assume $s>0$, $1\le p\le \infty$, $0<\tau\le\infty$, $m>0$,
and let $\Phi$ be a~Littlewood-Paley function as above.
In the setting described above and with $\widehat{\Phi}_\delta$ from \eqref{def-Phi-kde} we have:

$(i)$ If $2\le p<\infty$ and $\delta = n^{-\frac 1{2s+d}}$, then 
\begin{equation}\label{kde-1}
\sup_{f\in B^s_{p\tau}(m)}\bE \|\widehat{\Phi}_\delta-f\|_p
\le c n^{-\frac{s}{2s+d}},
\end{equation}
where $c=c(p,s,m,\tau)>0$.

$(ii)$ If $1\le p <2$, $x_0\in \MM$, $R>0$, and  $\delta = n^{-\frac 1{2s+d}}$,
then 
\begin{equation}\label{kde-2}
\sup_{f\in B^s_{p\tau}(m,x_0,R)}\bE \|\widehat{\Phi}_\delta-f \|_p
\le cn^{-\frac{s}{2s+d}},
\end{equation}
where $c=c(p,s,m,\tau,x_0,R)>0$.

$(iii)$ 
If $\delta = \big(\frac{\log n}n\big)^{\frac 1{2s+d}}$,
then 
\begin{equation}\label{kde-3}
\sup_{f\in B^s_{\infty \tau}(m)}\bE \| \widehat{\Phi}_\delta-f \|_\infty
\le c \Big(\frac{\log n}{ n}\Big)^{\frac s{2s+d}},
\end{equation}
where $c=c(s,m,\tau)>0$.
\end{theorem}

\begin{remarks}\label{minimaxity}
- Note that we do not claim  that the above rates are  necessarily minimax, although they show similarities with the results established in $\bR^d$.
The length of the paper does not allow to investigate the full lower bounds results. Let us just mention that if we add in this setting the Ahlfors condition {\bf C1A}, then  the lower bounds can be obtained in accordance with the upper bounds using a proof which is a direct adaptation of the one given in the case of the sphere in \cite{BKMP}. In the case where the Ahlfors condition is not valid, the problem is more complex since not only the regularity might be non-homogeneous  due to Besov conditions but also the dimension itself may vary spatially. In this case, the upper bounds might not be optimal. \\
- It is interesting to compare the obtained upper bounds for $\MM=[-1,1]$ in the different cases (Torus or Jacobi). In the Torus case, no surprise: the rate is   the usual one with dimension $d=1$. In the Jacobi case, the dimension is $d=1+(2\alpha+1)_+\vee(2\beta+1)_+$, which in particular in the case $\alpha=\beta=0$ (corresponding to $\mu$ the Lebesgue measure), gives a slower rate than the usual one. This is due to the fact that the 'boundaries' are playing a role so the spaces of approximation are not the same. The case $\alpha=\beta=-\frac12$, which was corresponding to perfect identification with the semi-circle (see Remark \ref{boundaries}) provides the expected rate with dimension $d=1$.
\end{remarks}

We next compile some additional facts we need about kernels in the setting of this article
and then carry out the proof of Theorems~\ref{thm:oracle-ineq} and \ref{thm:upper-bound}.

\subsection{Spectral multiplier integral operators}\label{sec:multipliers}

The operator $\Phi(\delta \sqrt{L})$ and its symmetric kernel $\Phi(\delta \sqrt L)(x,y)$ from above
have a number of useful properties that we describe and prove next.

\smallskip

(a) For any $k>d$ there exists a constant $c_k>0$ such that
\begin{equation}\label{PHI2}
|\Phi(\delta \sqrt L)(x,y)| \leq c(k)|B(x,\delta)|^{-1}\big(1+ \delta^{-1}\rho(x,y)\big)^{-k},
\quad x,y\in \MM, \; 0<\delta\le 1,
\end{equation}
where the constant $c(k)>0$ depends only on $k$, $\Phi$, and constant from the seting in Section~\ref{sec:setting}.
This inequality follows immediately from Theorem~\ref{thm:S-local-kernels}.

\smallskip

(b) For any $1\le p  \le \infty$
\begin{equation}\label{PHI1}
\|\Phi(\delta \sqrt L)(x,\cdot)\|_p \le c|B(x,\delta)|^{\frac 1p-1}
\le c_\star\delta^{-d(1-\frac 1p)},
\quad x\in \MM, \; 0<\delta\le 1,
\end{equation}
where the constant $c_\star>0$ is independent of $p$.
This estimate follows readily by \eqref{PHI2}, \eqref{tech-1}, and \eqref{DNC},
just as the estimates in \eqref{norm-r}.

\smallskip

(c) Let $X$ be a random variable on $\MM$ and $X \sim f(u) d\mu(u)$. Then
\begin{equation}\label{PHI4}
\bE \big(\Phi(\delta \sqrt L)(x,X)\big)
= \int_\MM \Phi(\delta \sqrt L)(x,u) f(u) d\mu(u)
= \Phi(\delta \sqrt L)f(x),
\quad x\in \MM.
\end{equation}
This is a well known property of expected values.

\smallskip

We next estimate the bias term of the risk.

\begin{proposition}\label{prop:approx}
Let $s>0$, $1\le p \le \infty$, $0<q\le\infty$.
If $f\in B^s_{pq}$, then $f\in \bL^p$ and
\begin{equation}\label{PHI3}
\|\Phi(\delta \sqrt L)f -f \|_p \le c\delta^s \|f\|_{B^s_{pq}},
\quad 0<\delta\le 1,
\end{equation}
where $c=c(s,p,q)>0$.
\end{proposition}

This statement is quite standard. For completeness we give its proof in the appendix.

\smallskip

We will also need the following two lemmas:

\begin{lemma}\label{lem:prep}
Let $2\le p<\infty$ and $0<\delta\le1$.
Then for any pdf $f$ on $\MM$ we have
\begin{equation}\label{prep-1}
\Big(\int_\MM\int_\MM|\Phi(\delta\sqrt{L})(x,u)|^p f(u)d\mu(u)d\mu(x) \Big)^{1/p}
\le c_\star\delta^{-d(1-1/p)}
\end{equation}
and
\begin{equation}\label{prep-2}
\Big(\int_\MM\Big(\int_\MM|\Phi(\delta\sqrt{L})(x,u)|^2 f(u)d\mu(u)\Big)^{p/2}d\mu(x) \Big)^{1/p}
\le c_\star\delta^{-d/2}\|f\|_{p/2}^{1/2},
\end{equation}
where $c_\star>0$ is the constant from \eqref{PHI1}; $c_\star$ is independent of $p$.
\end{lemma}

\begin{proof}
Denote by $S_1$ the quantity on the left-hand side in \eqref{prep-1}.
To estimate $S_1$ we use Fubini's theorem, (\ref{PHI1}), and the fact that $\int_\MM f(u) d\mu(u)=1$.
We obtain
\begin{equation}
\label{Ptl5}
S_1= \Big(\int_\MM\|\Phi(\delta\sqrt{L})(\cdot,u)\|_p^p f(u)d\mu(u)\Big)^{\frac 1p}
\le c_\star \delta^{-d(1-1/p)}.
\end{equation}
which confirms \eqref{prep-1}.

\smallskip

Let $S_2$ denote the quantity on the left-hand side in \eqref{prep-2} and
consider the integral operator $T$ with kernel
\begin{equation*}
T(x,y):=|\Phi(\delta\sqrt{L})(x,y)|^2.
\end{equation*}
By \eqref{PHI1} it follows that
$\|T(x,\cdot)\|_1=\|T(\cdot,y)\|_1\le c_\star^2\delta^{-d}$.
Therefore, by Schur's lemma (see e.g. \cite[Theorem 6.36]{Folland}) we obtain
\begin{equation}\label{Ptl6}
S_2^2 = \|Tf\|_{p/2} \le c_\star^2 \delta^{-d}\|f\|_{p/2}
\end{equation}
and inequality \eqref{prep-2} follows.
\end{proof}

\begin{lemma}\label{lemmaforlinear}
Let $1\le p<2$. Then there exists a constant $c=c(p)>0$ such that
for any $\delta>0$ and
any pdf $f$ supported in a ball $B(x_0,R)$ with $x_0\in \MM$ and $R\ge\delta/2$
we have
\begin{equation}\label{PHI5}
\Big(\int_\MM\Big(\int_\MM|\Phi(\delta\sqrt{L})(x,u)|f(u)d\mu(u)\Big)^{p/2}d\mu(x) \Big)^{\frac 1p}
\le c|B(x_0,R)|^{\frac{1}{p}-\frac{1}{2}}.
\end{equation}
\end{lemma}

\begin{proof}
We split the region of integration $\MM$ into two: $B(x_0,2R)$ and $\MM\setminus B(x_0,2R)$.

Since $1\le p<2$, there exists $1<r<\infty$ such that $\frac{p}{2}+\frac{1}{r}=1$.
Applying H\"{o}lder's inequality we obtain
\begin{align*}
I&:=\int_{B(x_0,2R)}\Big(\int_{B(x_0,R)}|\Phi(\delta\sqrt{L})(x,u)|f(u)d\mu(u)\Big)^{p/2}d\mu(x)
\\
&\le\Big(\int_{B(x_0,2R)}\int_{B(x_0,R)}|\Phi(\delta\sqrt{L})(x,u)|f(u)d\mu(u)d\mu(x)\Big)^{p/2}|B(x_0,2R)|^{1/r}.
\end{align*}
We now use Fubini's theorem, (\ref{PHI1}) with $p=1$, the fact that $\int_\MM f(u)d\mu(u)=1$,
and the doubling property (\ref{doubl-0}) to obtain
\begin{equation}\label{LFL1}
I\le \Big(\int_{B(x_0,R)}\|\Phi(\delta\sqrt{L})(\cdot,u)\|_1f(u)d\mu(u)\Big)^{p/2}|B(x_0,2R)|^{1/r}
\le c|B(x_0,R)|^{1-\frac{p}{2}}.
\end{equation}

To estimate the integral over $\MM\setminus B(x_0,2R)$
we observe that if $u\in B(x_0,R)$ and $x\in \MM\setminus B(x_0,2R)$,
then $\rho(x,x_0)\le\rho(x,u)+\rho(u,x_0)<\rho(x,u)+R<2\rho(x,u)$.
Then by \eqref{D2} and \eqref{PHI2} we get
\begin{align*}
|\Phi(\delta\sqrt{L})(x,u)|&\le c|B(x,\delta)|^{-1}\big(1+\delta^{-1}\rho(x,u)\big)^{-k}
\\
&\le c|B(x_0,\delta)|^{-1}\big(1+\delta^{-1}\rho(x,x_0)\big)^{-k+d}.
\end{align*}
Choose $k>d(1+2/p)$. Then using (\ref{tech-1}) we obtain
\begin{align}\label{LFL2}
&\int_{\MM\setminus B(x_0,2R)}\Big(\int_{B(x_0,R)}|\Phi(\delta\sqrt{L})(x,u)|f(u)d\mu(u)\Big)^{p/2}d\mu(x)\nonumber
\\
&\le c|B(x_0,\delta)|^{-p/2}\int_\MM \big(1+\delta^{-1}\rho(x,x_0)\big)^{-(k-d)p/2}d\mu(x)\Big(\int_{\MM}f(u)d\mu(u)\Big)^{p/2}
\\
&\le c|B(x_0,\delta)|^{1-p/2}
\le c|B(x_0,2R)|^{1-p/2}
\le c'|B(x_0,R)|^{1-p/2},\nonumber
\end{align}
where we used that $R\ge\delta/2$, $p<2$, and $\int_{\MM}f(u)d\mu(u)=1$.

Inequality \eqref{PHI5} follows by \eqref{LFL1} and \eqref{LFL2}.
\end{proof}

\subsection{Proof of Theorem~\ref{thm:oracle-ineq} and Theorem~\ref{thm:upper-bound}}\label{sec:Proof-ker-est}

We will only prove Theorem~\ref{thm:oracle-ineq}. Theorem~\ref{thm:upper-bound} follows readily.

By the triangle inequality we obtain the standard decomposition of the risk as the sum of stochastic and bias terms:
\begin{equation}\label{Ptl1}
\bE \| \widehat{\Phi}_\delta-f \|_p \le \bE\|\widehat{\Phi}_\delta-\Phi(\delta\sqrt{L})f\|_p +\|\Phi(\delta\sqrt{L})f-f \|_p.
\end{equation}
For the estimation of the bias term $\|\Phi(\delta\sqrt{L})f-f \|_p$
we will use estimate \eqref{PHI3}.
We next focus on the estimation of the stochastic term $\bE\|\widehat{\Phi}_\delta-\Phi(\delta\sqrt{L})f\|_p$.
In the case $1 \le p <\infty$, using Jensen's inequality, we get
\begin{align}
\bE( \| \widehat{\Phi}_\delta-\Phi(\delta\sqrt{L})f \|_p) &\leq  \big(\bE \| \widehat{\Phi}_\delta-\Phi(\delta\sqrt{L})f  \|^p_p\big)^{\frac 1p}
\nonumber
\\
&=
\Big( \int_\MM \bE \Big|\frac 1n \sum_{i=1}^n \Phi(\delta \sqrt L)(x,X_i)-\Phi(\delta \sqrt L)f(x)\Big|^pd\mu(x)\Big)^{\frac 1p}.
\label{Ptl2}
\end{align}

(i) Assume the pdf $f\in B^s_{p\tau}(m)$ and let $X\sim X_i$.
We first prove estimate \eqref{kde-1} for $p=2$.
Clearly
\begin{align*}
\bE \Big|\frac 1n \sum_{i=1}^n \Phi(\delta \sqrt L)(x,X_i)-\Phi(\delta \sqrt L)f(x)\Big|^2
&\le \frac{1}{n}\bE[\Phi(\delta \sqrt L)(x,X)]^2
\\
&= \frac{1}{n}\int_\MM|\Phi(\delta\sqrt{L})(x,u)|^2f(u)d\mu(u).
\end{align*}
This coupled with (\ref{Ptl2}) yields
\begin{align}\label{Ptl3}
\bE\| \widehat{\Phi}_\delta &-\Phi(\delta\sqrt{L})f\|_2
\\
&\le \frac{1}{n^{1/2}}\Big(\int_\MM\int_\MM|\Phi(\delta\sqrt{L})(x,u)|^2f(u)d\mu(u)d\mu(x) \Big)^{\frac 12}
\le \frac{c}{(n\delta^d)^{1/2}}, \nonumber
\end{align}
where we used \eqref{prep-1} with $p=2$.
Combining (\ref{Ptl1}), (\ref{PHI3}), and (\ref{Ptl3}) we get
\begin{equation*}
\bE \| \widehat{\Phi}_\delta-f \|_2 \le \frac{c}{(n\delta^d)^{1/2}} + cm\delta^s.
\end{equation*}
With  $\delta=n^{-\frac{1}{2s+d}}$, i.e. $\delta^s=\frac{1}{(n\delta^d)^{1/2}}$,
this yields \eqref{kde-1} when $p=2$.

\smallskip

Let $2< p<\infty$. We will use the following version of Rosenthal's inequality that can be derived
for instance from \cite{HKPT}, p. 245, inequality (C.5) with $\tau= \frac P2+1 \leq p+1$:
If $Y_1,\dots,Y_n$ are i.i.d. random variables and $Y_i\sim Y$, then
\begin{equation}\label{Rosenthal}
\bE \Big|\frac 1n\sum_{i=1}^n Y_i -\bE Y \Big|^p
\le \frac{(p+1)^p}{n^{p-1}}\bE |Y|^p
+ \frac{p(p +1)^{p/2}e^{p/2+1}}{n^{p/2}}\big(\bE |Y|^2\big)^{p/2}.
\end{equation}


We get
\begin{align*}
\bE \Big|\frac 1n \sum_{i=1}^n \Phi(\delta \sqrt L)(x,X_i)&-\Phi(\delta \sqrt L)f(x)\Big|^p
\\
&\le \frac c{n^{p-1}} \bE\big|\Phi(\delta\sqrt{L})(x,X)\big|^p
+ \frac c{n^{p/2}} \Big(\bE\big|\Phi(\delta\sqrt{L})(x,X)\big|^2\Big)^{p/2}.
\end{align*}
This and \eqref{Ptl2} imply
\begin{align}\label{Ptl4}
\bE \| \widehat{\Phi}_\delta-\Phi(\delta\sqrt{L})f\|_p
&\le \frac c{n^{1-1/p}} \Big(\int_\MM\int_\MM|\Phi(\delta\sqrt{L})(x,u)|^p f(u)d\mu(u)d\mu(x) \Big)^{1/p} \nonumber
\\
&+ \frac c{n^{1/2}}\Big(\int_\MM\Big(\int_\MM|\Phi(\delta\sqrt{L})(x,u)|^2 f(u)d\mu(u)\Big)^{p/2}d\mu(x) \Big)^{1/p}
\\
&= \frac{c}{(n\delta^d)^{1/2}}+ c(n\delta^{d})^{-1}\|f\|_{p/2}, \nonumber
\end{align}
where we used \eqref{prep-1} and \eqref{prep-1}.
Since $1\le\frac{p}{2}<p$ and $\|f\|_1=1$, we obtain by interpolation
\begin{equation}\label{Ptl7}
\| f \|_{\frac p2} \le \|f \|_1^{\frac 1{p-1}} \| f \|_p^{\frac{p-2}{p-1}}
=\| f \|_p^{\frac{p-2}{p-1}}
\le c\|f\|_{B^s_{p\tau}}^{\frac{p-2}{p-1}}
\le cm^{\frac{p-2}{p-1}}.
\end{equation}
Here we also used Proposition \ref{Pr:Embed} (iv).

Combining \eqref{Ptl4}-\eqref{Ptl7} with \eqref{Ptl1} and \eqref{PHI3},
and taking into account that
$\delta=n^{-\frac{1}{2s+d}}$, i.e. $\delta^s=\frac{1}{(n\delta^d)^{1/2}}$
we arrive at
\begin{equation}\label{est-E1}
\bE \| \widehat{\Phi}_\delta-f \|_p \le \frac{c}{(n\delta^d)^{1/2}} + cm\delta^s
\le c'n^{-\frac{s}{2s+d}}.
\end{equation}
The proof of part (i) of the theorem is complete.

\medskip

(ii) Let $1\le p< 2$ and $f\in B^{s}_{p\tau}(m,x_0,R)$.
We use Jensen's inequality and the fact that
$|\Phi(\delta \sqrt L)(x,u)|\le c|B(x, \delta)|^{-1} \le c'\delta^{-d}$,
using (\ref{PHI2}) and (\ref{DNC}), to obtain
\begin{align*} 
\bE \Big|\frac 1n \sum_{i=1}^n \Phi(\delta \sqrt L)(x,X_i)&-\Phi(\delta \sqrt L)f(x)\Big|^p
\\
&\le \Big(\bE \Big|\frac 1n \sum_{i=1}^n \Phi(\delta \sqrt L)(x,X_i)-\Phi(\delta \sqrt L)f(x)\Big|^2\Big)^{\frac p2} 
\\
&\le\frac 1{ n^{\frac p2}} \big(\bE |\Phi(\delta \sqrt L)(x,X)|^2\big)^{\frac p2}
\\
&=\frac 1{ n^{\frac p2}} \Big(\int_\MM |\Phi(\delta \sqrt L)(x,u)|^2f(u)d\mu(u)\Big)^{\frac p2} 
\\
&\le \frac c{ (n\delta^d)^{\frac p2}} \Big(\int_\MM |\Phi(\delta \sqrt L)(x,u)|f(u)d\mu(u)\Big)^{\frac p2}. 
\end{align*}
This and \eqref{Ptl2} lead to
\begin{align*}
\bE \| \widehat{\Phi}_\delta-\Phi(\delta\sqrt{L})f\|_p
\le \frac c{(n\delta^d)^{\frac 12}}
\Big(\int_\MM\Big(\int_\MM |\Phi(\delta \sqrt L)(x,u)|f(u)d\mu(u)\Big)^{\frac p2}d\mu(x)\Big)^{\frac{1}{p}}.
\end{align*}
We now invoke Lemma \ref{lemmaforlinear} to obtain
\begin{align*}
\bE \| \widehat{\Phi}_\delta-\Phi(\delta\sqrt{L})f\|_p
\le \frac c{(n\delta^d)^{\frac 12}}|B(x_0, R)|^{\frac{1}{p}-\frac{1}{2}}
\le \frac{c'}{(n\delta^d)^{\frac 12}}.
\end{align*}
Using this and \eqref{PHI3} we complete the proof of \eqref{kde-2}
just as above in \eqref{est-E1}.

\medskip

(iii) Assume the pdf $f\in B^s_{\infty \tau}(m)$ and let  $q>2$ be arbitrary.
Since by construction $\supp \Phi\subset [-1,1]$,
the function $\widehat{\Phi}_\delta(x)-\Phi(\delta \sqrt L)f(x)$
belongs to the spectral space $\Sigma_{1/\delta}$.
Then by Proposition~\ref{prop:nikolski}
\begin{equation*}
\| \widehat{\Phi}_\delta-\Phi(\delta \sqrt L)f\|_\infty
\le c_\star \delta^{-\frac dq}\| \widehat{\Phi}_\delta-\Phi(\delta \sqrt L)f\|_q,
\end{equation*}
where the constant $c_\star>1$ is independent of $q$.
This along with Jensen's inequality and Fubini's theorem lead to
\begin{equation}\label{est-E-inf}
\bE \| \widehat{\Phi}_\delta-\Phi(\delta \sqrt L)f \|_\infty
\le c_\star\delta^{-\frac dq}\Big(\int_\MM \bE | \hat{\Phi}_\delta(x,\cdot)-\Phi(\delta \sqrt L)f(x)|^q d\mu(x)\Big)^{\frac 1q}.
\end{equation}
We now apply
 Rosenthal's inequality (\ref{Rosenthal}) to obtain
\begin{align*}
\bE | \widehat{\Phi}_\delta(x)&-\Phi(\delta \sqrt L)f(x)|^q
\\
&\le \frac{(q+1)^q}{n^{q-1}}\bE |\Phi(\delta\sqrt L)(x,X)|^q
+ \frac{q(q +1)^{\frac q2}e^{\frac q2 +1}}{n^{\frac q2}}
\big(\bE |\Phi(\delta \sqrt L)(x,X)|^2\big)^{\frac q2}
\\
&=\frac{(q+1)^q}{n^{q-1}}\int_\MM |\Phi(\delta \sqrt L)(x,u)|^q f(u)d\mu(u)
\\
&+ \frac{q(q +1)^{\frac q2}e^{\frac q2 +1}}{n^{\frac q2}}
\Big(\int_\MM|\Phi(\delta \sqrt L)(x,u)|^2 f(u)d\mu(u)  \Big)^{\frac q2}.
\end{align*}
This coupled with \eqref{est-E-inf} and the fact that $1/q<1$ imply
\begin{align}\label{infty}
\bE\| \hat{\Phi}_\delta &-\Phi(\delta\sqrt L)f\|_\infty
\nonumber
\\
&\le c_\star \delta^{-\frac dq} \frac{q+1}{n^{1-\frac{1}{q}}}
\Big(\int_\MM \int_\MM |\Phi(\delta \sqrt L)(x,u)|^q f(u)d\mu(u)d\mu(x)\Big)^{\frac 1q}
\nonumber\\
&+ c_\star \delta^{-\frac dq}\frac{q^{\frac{1}{q}}(q +1)^{\frac 12}e^{\frac 12 +\frac 1q}}{n^{\frac 12}}
\Big(\int_\MM \Big(\int_\MM|\Phi(\delta \sqrt L)(x,u)|^2 f(u)d\mu(u)  \Big)^{\frac q2}d\mu(x)\Big)^{\frac 1q}
\\
&\le c_\star \delta^{-\frac dq}\Big(\frac{2c_\star q}{(n\delta^d)^{1-\frac{1}{q}}}
+ \frac{e^2 c_\star q^{1/2}}{(n\delta^d)^{1/2}}\|f\|_{q/2}^{1/2}\Big),
\nonumber
\end{align}
where we used \eqref{prep-1}, \eqref{prep-2},
and the inequality
$q^{\frac{1}{q}}(q +1)^{\frac 12}e^{\frac 12 +\frac 1q} \le e^2q^{1/2}$, ($q>2$).
Observe that the constant $c_\star$ above is from \eqref{PHI1} and is independent of $q$.

By Proposition~\ref{Pr:Embed} (iii) it follows that  $f\in \bL^{\infty}$
and since $\|f\|_1=1$ we obtain
$$
\|f\|_{q/2}\le\|f\|_{\infty}^{1-2/q}\|f\|_1
\le \big(c\|f\|_{B^s_{\infty\tau}}\big)^{1-2/q}
\le (cm)^{1-2/q} \le cm+1.
$$

Let $n\ge e^2$ and choose $q:=\log n$. By assumption $\delta=\big(\frac{\log n}{n}\big)^{1/(2s+d)}$.
Now, it is easy to see that $n^{1/q}=e$, $\delta^{-d/q}\le e$,
$\delta^s=\frac{q^{1/2}}{(n\delta^d)^{1/2}}=\big(\frac{\log n}{n}\big)^{s/(2s+d)}$,
and
\begin{equation*}
\frac{q}{(n\delta^d)^{1-1/q}} \le \frac{q}{(n\delta^d)^{3/4}}
\le \frac{\log n}{n^{\frac{3s/2}{2s+d}}}
\le c\Big(\frac{\log n}{n}\Big)^{s/(2s+d)}
\quad\hbox{if}\quad
n\ge e^2.
\end{equation*}
Putting all of the above together we obtain
\begin{equation}\label{est-E-infty}
\bE\| \hat{\Phi}_\delta -\Phi(\delta\sqrt L)f\|_\infty
\le c \Big(\frac{\log n}{n}\Big)^{s/(2s+d)}.
\end{equation}

If $2\le n<e^2$, then estimate \eqref{est-E-infty} follows readily from (\ref{infty}) with $q=2$.

As before we use \eqref{est-E-infty} and \eqref{PHI3} to obtain \eqref{kde-3}.
The proof of Theorem~\ref{thm:upper-bound} is complete.

A closer examination of the above proof
shows that the oracle inequalities from Theorem~\ref{thm:oracle-ineq} are valid.
\qed

\subsection{Linear wavelet density estimators}\label{sec:linear-wavelet-est}

In this section we establish $\bL^p$-error estimates for linear wavelet density estimators.
Let $\{\psi_{j\xi}\}$, $\{\tilde{\psi}_{j\xi}\}$ be the pair of dual frames described in Subsection~\ref{Frames}.
We adhere to the notation from Section~\ref{Frames}.

For any $j\ge 0$ and $\xi\in\cX_j$ we define the \textit{empirical coefficient estimators} by
\begin{equation}\label{ee}
\hat{\beta}_{j\xi}:=\frac{1}{n}\sum_{i=1}^n \tilde{\psi}_{j\xi}(X_i).
\end{equation}
Using this we define the \textit{linear wavelet density estimator} by
\begin{equation}\label{flin}
f^*(x)=\sum_{j=0}^J\sum_{\xi\in\cX_j}\hat{\beta}_{j\xi}\psi_{j\xi}(x),\;\;x\in \MM,
\end{equation}
where 
the parameter $J=J(n)\in\mathbb{N}$ is selected so that the factor $b^{-J}$
de facto behaves as a bandwidth.
More precisely, we define $J$ as the unique positive integer such that
\begin{equation}\label{bJlin}
b^J\le n^{1/(2s+d)}<b^{J+1}.
\end{equation}
It is easy to see that $f^*$ can be written in the following way
\begin{align}\label{flin2}
f^*(x)&=\frac{1}{n}\sum\limits_{i=1}^n\sum\limits_{j=0}^J\sum_{\xi\in\cX_j}\psi_{j\xi}(x)\tilde{\psi}_{j\xi}(X_i)
\\
&=\frac{1}{n}\sum\limits_{i=1}^n\sum\limits_{j=0}^J\Psi_j(\sqrt{L})(X_i,x)
=\frac{1}{n}\sum\limits_{i=1}^n\Psi_0(b^{-J}\sqrt{L})(X_i,x). \nonumber
\end{align}
where we used \eqref{frame1} and \eqref{PsiJ}.

Thus, this linear wavelet estimator is in fact a particular case of
the linear estimators investigated in the previous subsection.
This enables us to state the following upper bound theorem,
which is an immediate consequence of Theorem~\ref{thm:upper-bound}.

\begin{theorem}\label{Th:Lin}
Let $s>0$, $0<\tau\le\infty$, $m>0$, $x_0\in \MM$ and $R>0$.

$(i)$ If $2\le p < \infty$ and $J$ is as in $(\ref{bJlin})$, then
\begin{equation}\label{UpL}
\sup_{f\in B^s_{p\tau}(m)} \bE \|f^*-f\|_p\le c n^{-s/(2s+d)},
\end{equation}
where $c=c(p,\tau,s,m)>0$.

$(ii)$ If $1\leq p <2$ and $J$ is as in $(\ref{bJlin})$, then
\begin{equation}\label{UpL2}
\sup_{f\in B^s_{p\tau}(m,x_0,R)} \bE \|f^*-f\|_p\le c n^{-s/(2s+d)},
\end{equation}
where $c=c(p,\tau,s,m,x_0,R)>0$.

$(iii)$ 
If $J$ is the unique integer satisfying
$b^J\leq \big(\frac n{\log n}\big)^{1/(2s+d)}<b^{J+1}$,
then
\begin{equation}\label{UpL3}
\sup_{f\in B^s_{\infty\tau}(m)}\bE \|f^*-f \|_\infty \le c \Big(\frac{\log n}{ n}\Big)^{\frac s{2s+d}},
\end{equation}
where $c=c(\tau,s,m)>0$.
\end{theorem}

\section{Adaptive wavelet density estimation by thresholding}\label{Adaptive}

If we want to parallel the achievements obtained in density estimation theory in -say- $[0,1]^d$,
one important feature is lacking: adaptation (i.e. obtaining -up to logarithmic factors-
optimal rates of convergence without knowing the regularity).
There are various techniques for this.
For instance, Lepski's method (see \cite{Lepski, Oleg}) could be applied to our kernel estimators.

We choose to develop here nonlinear wavelet estimators, where we apply \textit{hard thresholding}.
This method has been developed in the classical case of $\R$ in \cite{DJ} and on the sphere in \cite{HKPT}.
We will operate in the general setting described in Section~\ref{sec:setting}.
Unlike the case of the kernel or linear wavelet density estimates considered in the previous section,
here we assume that the space $\MM$ is compact ($\mu(\MM) <\infty$)
and all conditions {\bf C1}--{\bf C5}
(including the Ahlfors regularity condition {\bf C1A}) are satisfied, see Section~\ref{sec:setting}.

As before we assume that
$X_1,\dots,X_n$ ($n\ge 2$) are i.i.d. random variables with values on~$\MM$
and with a common density function $f$ with respect to the measure $\mu$ on $\MM$.
Let $X_j\sim X$.
We denote by $\bE=\bE_f$ the expectation with respect to
the probability measure $\bP=\bP_f$.

In addition, we assume that $f$ is bounded.
Denote
\begin{equation}\label{def-A}
A:=\max\big\{\|f\|_\infty,4\big\}
\quad\hbox{and set}\quad
\kappa:=\cd(8A)^{1/2},
\end{equation}
where $\cd>1$ is the constant from the norm bounds of the frame elements in \eqref{bg2}.

We will utilize the pair of frames $\{\psi_{j\xi}\}$, $\{\tilde{\psi}_{j\xi}\}$ described in \S\ref{Frames}.
We adhere to the notation from \S\ref{Frames}.
Recall that any $f\in\bL^p(\MM, d\mu)$ has the frame decomposition
\begin{equation}\label{frame-decomp}
f=\sum\limits_{j=0}^{\infty}\sum_{\xi\in\cX_j}\beta_{j\xi}(f)\psi_{j\xi},
\quad
\beta_{j\xi}(f):=\langle f,\tilde\psi_{j\xi}\rangle
\quad\hbox{(convergence in $\bL^p$).}
\end{equation}
Assuming the pdf $f$ fixed, we will use the abbreviated notation $\beta_{j\xi}:=\beta_{j\xi}(f)$.

We introduce two parameters depending on $n$:
\begin{equation}\label{tn}
\lambda_n:=\kappa\Big(\frac{\log n}{n}\Big)^{1/2}
\end{equation}
and $J_n$ uniquely defined by the following inequalities
\begin{equation}\label{bJ}
b^{J_n}\leq \big(\frac n{\log n}\big)^{1/d}<b^{J_n+1}.
\end{equation}

As in \S\ref{sec:linear-wavelet-est} we introduce the \textit{empirical coefficient estimators}
\begin{equation}\label{lin-coef-est}
\hat{\beta}_{j\xi}:=\frac{1}{n}\sum_{i=1}^n \tilde{\psi}_{j\xi}(X_i),
\quad j\ge 0, \; \xi\in\cX_j.
\end{equation}
We now define the \textit{hard threshold coefficient estimators} by
\begin{equation}\label{hte}
\hat{\beta}_{j\xi}^*:=\hat{\beta}_{j\xi}\ONE_{\{|\hat{\beta}_{j\xi}|>2\lambda_n\}},
\quad j\ge 0, \; \xi\in\cX_j.
\end{equation}
Then the \textit{wavelet threshold density estimator} is defined by
\begin{equation}\label{WDE}
\hat{f}_n(x):=\sum_{0\le j \le J_n}\sum_{\xi\in\cX_j}\hat{\beta}_{j\xi}^*\psi_{j\xi}(x),
\quad x\in\MM.
\end{equation}

\begin{remark}
Note that the density estimator $\hat{f}_n$ of the pdf $f$ depends only on the number $n$ of observations,
the geometric constant $\cd$, and the $\bL^{\infty}$-norm of $f$.
\end{remark}

We now state our main result on the adaptive wavelet threshold estimator defined above.


\begin{theorem}\label{Tresh}
Let $1\le r\le\infty$, $0<\tau\le\infty$, $1\le p<\infty$, $s>d/r$, and $m>0$.
Then there exists a constant $c=c(r,\tau,p,s,m)>0$ such that
in the setting described above and with $\hat{f}_n$ from $\eqref{WDE}$ we have:

$(i)$
\begin{equation}\label{main-1}
\sup_{f\in B^s_{r\tau}(m)}\bE \|\hat{f}_n -f\|_\infty
\le c\Big(\frac{\log n}{n}\Big)^{\frac{s-\frac{d}{r}}{2[s-d(\frac{1}{r}-\frac{1}{2})]}}.
\end{equation}

$(ii)$ In the regular case $s \ge \frac{dp}{2}\big(\frac{1}{r}- \frac{1}{p}\big)$
\begin{equation}\label{main-2}
\sup_{f\in B^s_{r\tau}(m)}\bE\|\hat{f}_n -f\|_p
\le c \log n \Big(\frac{\log n}{n}\Big)^{\frac {s}{2s+d}}.
\end{equation}

$(iii)$ In the sparse case $s <\frac{dp}{2}\big(\frac{1}{r}- \frac{1}{p}\big)$
\begin{equation}\label{main-3}
\sup_{f\in B^s_{r\tau}(m)}\bE\|\hat{f}_n -f\|_p
\le c\log n\Big(\frac{\log n}{n}\Big)^{\frac{s-d(\frac{1}{r} -\frac{1}{p})}{2[s- d(\frac{1}{r} -\frac{1}{2})]}}.
\end{equation}

\end{theorem}

\begin{remark}
Several observations are in order:

$(a)$
The assumption $s>d/r$ leads to $\|f\|_{\infty}\le c\|f\|_{B^s_{r\tau}}\le cm$,
by Proposition~\ref{Pr:Embed}~$(iii)$. In addition in the sparse case it implies $p>2$.

$(b)
$ The geometry of the setting is represented by the dimension $d$.
Note that the exponents of $\frac{\log n}{n}$ are the same as in the case of the sphere \cite{BKMP}.

$(c)$
In the regular case (modulo the logarithmic terms) we have the same
rate of convergence $n^{-s/(2s+d)}$ as in the case of the linear wavelet estimator.

$(d)$ Just as in the case of kernel density estimators $($see Remark~\ref{minimaxity}$)$
we note that since we assume here all conditions {\bf C1-C5} $($including {\bf C1A}$)$
it would not be a problem to obtain lower bounds matching up to logarithmic terms
the rates established above by a direct adaptation of the proof of the lower bounds
in the case of the sphere from \cite{BKMP}.
\end{remark}

\subsection{Preparation for the proof of Theorem \ref{Tresh}}

We first recall the classical \textit{Bernstein inequality} (see e.g. \cite{Pollard}):
Let $Y_1,\dots,Y_n$ be independent random variables such that
$\bE Y_i=0$, $\bE Y_i^2 \le \sigma^2$, and $|Y_i|\le M$, $i=1,\dots,n$.
Then for any $v>0$
\begin{align}\label{Bernstein}
\bP\Big(\Big|\frac{1}{n}\sum_{i=1}^n Y_i\Big|\ge v\Big)
&\le 2\exp\Big(-\frac{nv^2}{2(\sigma^2+Mv/3)}\Big)
\le 2\exp\Big(-\frac{nv^2}{4\sigma^2} \wedge \frac{3nv}{4M}\Big)
\\
&= 2\exp\Big(-\frac{nv^2}{2\sigma^2}\Big) \ONE_{\{v\le \frac{ 3\sigma^2}{M}\}} +
 2\exp\Big(-\frac{3nv}{4M}\Big)\ONE_{\{v >\frac{ 3\sigma^2}{M}\}}. \notag
\end{align}

We next use  Rosenthal's inequalities \eqref{Rosenthal} to derive several
useful estimates in the current setting.
Clearly,
\begin{equation*}
\beta_{j\xi}:=\int_\MM\tilde{\psi}_{j\xi}(x)f(x) d\mu(x) = \bE (\tilde{\psi}_{j\xi}(X))
\end{equation*}
and using \eqref{bg2} we obtain
\begin{equation}\label{sigma}
\bE |\tilde{\psi}_{j\xi} (X) -\beta_{j\xi}|^2
\le \bE |\tilde{\psi}_{j\xi} (X)|^2
= \int_\MM f(x)|\tilde{\psi}_{j\xi}(x)|^2d\mu(x)
\le \cd^2A
\end{equation}
and
\begin{equation}\label{sup}
|\tilde{\psi}_{j\xi} (X) -\beta_{j\xi}|\leq 2 \|\tilde{\psi}_{j\xi} \|_\infty \leq 2 \cd b^{jd/2}.
\end{equation}
Also
\begin{equation*}
\bE|\tilde{\psi}_{j\xi} (X) -\beta_{j\xi}|^p \leq 2^p
\bE|\tilde{\psi}_{j\xi} (X)|^p=2^p\int_\MM f(x)|\tilde{\psi}_{j\xi} (x)|^p dx\mu(x)
\le (2\cd)^pA b^{jd(\frac p2-1)}.
\end{equation*}
Therefore, using \eqref{Bernstein}
\begin{equation}\label{Bern}
\bP \big(|\hat{\beta}_{j\xi} -\beta_{j\xi}|>\lambda\big)
\le 2 \exp{\Big(-\frac{n\lambda^2}{4\cd^2A} \wedge \frac{3n\lambda}{8\cd b^{j d/2}}\Big)}.
\end{equation}
From this with the notation
\begin{equation}\label{def-mj}
\mu_j := \frac{3}{2}\cd A b^{-j d/2}
\end{equation}
we obtain
\begin{equation}\label{Bern2}
\bP \big(|\hat{\beta}_{j\xi} -\beta_{j\xi}|>\lambda\big)
\le 2\exp\Big(- \frac{n\lambda^2}{4\cd^{2} A}\Big)
\quad\text{if}\quad 0\le \lambda \le \mu_j,
\end{equation}
and
\begin{equation}\label{Bern3}
\bP\big(|\hat{\beta}_{j\xi} -\beta_{j\xi}|>\lambda\big)
\leq 2 \exp\Big(- \frac{3n\lambda}{8\cd b^{j d/2}}\Big)
\quad\text{if}\quad \lambda \ge \mu_j.
\end{equation}
In particular, if $0\le j\le J_n$, then
\begin{equation}\label{Bern4}
\bP \big(|\hat{\beta}_{j\xi} -\beta_{j\xi}|>\lambda_n\big)
\le  \frac {2}{n^2}.
\end{equation}
Now, by (\ref{Rosenthal}), for any $p\ge2$ there exists $c=c(p)>0$ such that
\begin{equation}\label{Rosen1}
\bE |\hat{\beta}_{j\xi} -\beta_{j\xi}|^p
\le c\Big(\frac{\cd^2A}{n}\Big)^{p/2}\Big(1+2^p \Big(\frac{b^{jd}}{nA}\Big)^{p/2-1}\Big).
\end{equation}
Moreover, for any $j\ge0$ such that $b^{jd}\le nA$
\begin{equation}\label{Rosen3}
\bE|\hat{\beta}_{j\xi} -\beta_{j\xi}|^p
\le c\Big(\frac{\cd^2A}{n}\Big)^{p/2}.
\end{equation}
On the other hand, by Jensen inequality, for any $0<p\le2$
\begin{equation}\label{Rosen2}
\bE|\hat{\beta}_{j\xi} -\beta_{j\xi}|^p
\le \Big(\frac{cc_0^2A}{n}\Big)^{p/2}.
\end{equation}

The following two lemmas will be instrumental in the proof of Theorem~\ref{Tresh}.


\begin{lemma}\label{technical1}
For any $n\ge2$, $0\le j\le n$, and $\xi\in\cX_j$ we have
\begin{equation}\label{est-I-jxi}
I_{j\xi}:=\int_0^{\infty}
\bP\Big( |\hat{\beta}_{j\xi}- \beta_{j\xi}| \ONE_{\{|\hat{\beta}_{j\xi}- \beta_{j\xi}|>\lambda_n\} }>\lambda\Big) d\lambda
\le  cA^{1/2}\Big(\frac{\log n}{n}\Big)^{1/2}\frac 1{n^2}.
\end{equation}
where $c=c(\cd)>0$.
\end{lemma}

\begin{proof}
We will use the following well known inequality: 
\begin{equation}\label{integrals}
\int_a^\infty e^{-K\frac{\lambda^2}2} d\lambda
\le  \frac{e^{-\frac{Ka^2}{2}}}{K^{1/2}} \Big((\pi/2)^{1/2} \wedge \frac 1{aK^{1/2}}\Big),
\quad K, a>0.
\end{equation}
We split the integral in \eqref{est-I-jxi} into three:
\begin{align*}
I_{j\xi}
=\Big(\int_0^{\lambda_n}+\int_{\lambda_n}^{\mu_j}+\int_{\mu_j}^\infty\Big)
\bP\Big(|\hat{\beta}_{j\xi}- \beta_{j\xi}| \ONE_{\{|\hat{\beta}_{j\xi}- \beta_{j\xi}|>\lambda_n\} }>\lambda\Big) d\lambda
=: S_1+S_2+S_3.
\end{align*}
From the definitions in \eqref{tn} and \eqref{def-mj} it readily follows that $\mu_j\ge \lambda_n$.

To estimate $S_1$ we use \eqref{Bern2} and the definitions of $\lambda_n$, $\kappa$ in \eqref{tn}, \eqref{def-A}.
We get
\begin{align*}
S_1 \le \lambda_n \bP( |\hat{\beta}_{j\xi}- \beta_{j\xi}|>\lambda_n)
&\le \lambda_n 2 \exp\Big(-\frac{n\lambda_n^2}{4\cd^2A}\Big)
\\
&\le 2\kappa \Big(\frac{ \log n}n\Big)^{1/2} n^{- \frac{\kappa^2}{4\cd^2A}}
\le 2\kappa \Big(\frac{\log n}n\Big)^{1/2}n^{-2},
\end{align*}
where we used that $\frac{\kappa^2}{4c^2A}=2$.
For $S_2$ we use inequality \eqref{integrals} and again \eqref{Bern2} to obtain
\begin{align*}
S_2 \le \int_{\lambda_n}^{\infty} 2\exp\Big(-\frac{n\lambda^2}{4\cd^2A}\Big) d\lambda
&\le  2(\pi/2)^{1/2}\Big(\frac{2\cd^2A}{n}\Big)^{1/2}\exp\Big(-\frac{n\lambda_n^2}{4\cd^2A}\Big)
\\
&\le 6\cd A^{1/2}n^{-1/2} n^{- \frac{\kappa^2}{4\cd^2A}}
= 6\cd A^{1/2}n^{-5/2}.
\end{align*}
We now estimate $S_3$. Using \eqref{Bern3} we obtain
\begin{align*}
S_3 &\le \int_{\mu_j}^\infty 2 \exp\Big(-\frac{3n\lambda}{8\cd b^{j d/2}}\Big) d\lambda
= \frac{16\cd b^{j\frac d2}}{3n} \exp\Big(-\frac{3n\mu_j}{8\cd b^{j d/2}}\Big)
\\
&= \frac{16\cd b^{j\frac d2}}{3n} \exp\Big(-\frac{ 9A}{16} n b^{-jd}\Big)
\le \frac{6\cd}{(n \log n)^{1/2}} n^{-\frac{ 9A}{16}}
\le \frac{6\cd}{(n \log n)^{1/2}} n^{-9/4},
\end{align*}
where for the second equality we used the definition of $\mu_j$ in \eqref{def-mj},
for the former inequality we used that
$b^{-jd} \ge  b^{-J_nd} = \frac{\log n}n$,
and for the last inequality that $A\geq 4$.

Putting the above estimates for $S_1, S_2$, and $S_3$ together we arrive at
$$
I_{j\xi}
\le 2\kappa \Big(\frac{\log n}n\Big)^{1/2}n^{-2}
+ 6\cd A^{1/2}n^{-5/2}
+ \frac{6\cd n^{-9/4}}{(n \log n)^{1/2}}
\le cA^{1/2}\Big(\frac{\log n}{n}\Big)^{1/2}n^{-2}
$$
as claimed.
\end{proof}


\begin{lemma}\label{L1}
Let $\cF$ be a finite family of functions $\phi: \bR \to \bR$ with $N:=\card(\cF)$
and let $X$ be a random variable.
Assume
\begin{equation}\label{R0}
\sup_{\phi \in \cF}\|\phi\|_\infty \le M/2,
\quad\hbox{and}\quad
\sup_{ \phi \in \cF} \bE\big(\phi^2(X)\big)\le \sigma^2.
\end{equation}
Let $ X_1,\ldots, X_n$ be i.i.d. random variables and $X_i\sim X$.
Then
\begin{equation}\label{R1}
\bE\Big[\sup_{\phi \in \cF}  \Big| \frac 1n \sum_{i=1}^n \big(\phi(X_i)-\bE\big(\phi(X)\big)\big)\Big|\Big]
\le 4\sigma \Big(\frac{\log (2N)}n\Big)^{1/2} + 4M \frac{\log (2N)}{n} .
\end{equation}
\end{lemma}


\begin{proof}
By Bernstein's inequality \eqref{Bernstein} it follows that for any $\phi\in\cF$
\begin{equation}\label{Bern55}
\bP\Big(\Big| \frac 1n \sum_{i=1}^n \phi(X_i)-\bE [\phi(X)] \Big| \geq \lambda \Big)
\le 2 \exp\Big(-\frac{n\lambda^2}{4\sigma^2}\wedge \frac{ 3n\lambda}{4M}\Big).
\end{equation}
Let $\lambda_0:= \frac{3\sigma^2}{M}$ and
$Z _\phi:=  | \frac 1n \sum_{i=1}^n   (\phi(X_i)-\bE [\phi(X)])|.$
Then
$$
\sup_{\phi \in \cF} Z_\phi = \sup_{\phi \in \cF}\{Z_\phi \ONE_{\{Z_\phi \le \lambda_0\}}
+ Z_\phi\ONE_{\{Z_\phi > \lambda_0\}}\}
\le \sup_{\phi \in \cF} Z_\phi \ONE_{\{Z_\phi \leq \lambda_0\}}  +\sup_{\phi\in \cF}
Z_\phi \ONE_{\{Z_\phi > \lambda_0\}}  .
$$
From \eqref{Bern55} it follows that for any $\lambda >0$
$$
\bP (Z _\phi \ONE_{\{Z_\phi \leq \lambda_0\}} >\lambda)
\leq 2 e^{- \frac{n\lambda^2}{4\sigma^2} }
\quad\hbox{and}\quad
\bP(Z_\phi \ONE_{\{Z_\phi > \lambda_0\}} >\lambda)
\le 2 e^{-\frac{ 3n\lambda}{4 M}}.
$$
We use these inequalities and (\ref{integrals}) to obtain
\begin{align*}
\bE&\Big[\sup_{\phi \in \cF} Z_\phi\Big]
\le \bE \Big[\sup_{\phi \in \cF}  Z _\phi \ONE_{\{Z_\phi \leq \lambda_0\}}\Big]
+ \bE\Big[\sup_{\phi\in \cF} Z_\phi \ONE_{\{Z_\phi > \lambda_0\}}\Big]
\\
&= \int_0^{\infty} \bP \Big(\sup_{\phi \in \cF} Z_\phi \ONE_{\{Z_\phi \leq \lambda_0\}}>\lambda\Big) d\lambda
+  \int_0^{\infty} \bP\Big(\sup_{\phi \in \cF}  Z_\phi \ONE_{\{Z_\phi > \lambda_0\}}>\lambda\Big) d\lambda
\\
&\le  a + \sum_{\phi \in \cF} \int_a^{\infty} \bP\Big(  Z _\phi \ONE_{\{Z_\phi \leq \lambda_0\}}>\lambda\Big) d\lambda
+ b + \sum_{\phi \in \cF}  \int_b^{\infty} \bP\Big( Z _\phi \ONE_{\{Z_\phi > \lambda_0\}}>\lambda\Big) d\lambda
\\
&\le a + N \int_a^{\infty} 2 e^{- \frac{n\lambda^2}{4\sigma^2}}d\lambda
+ b + N \int_b^{\infty} 2 e^{-\frac{ 3n\lambda}{4 M}} d\lambda
\\
&\leq a +2N \frac{2\sigma^2}{na} e^{- \frac{na^2}{4\sigma^2} }+ b + 2N\frac{4M}{3n}e^{-\frac{ 3nb}{4 M}}.
\end{align*}
We now optimize with respect to $a$ and $b$ by taking
$2Ne^{-\frac{na^2}{4\sigma^2} } =1$ and $2N e^{-  \frac{ 3nb}{4 M}}=1$.
We obtain
\begin{align*}
\bE\Big[\sup_{\phi \in \cF} Z_\phi \Big]&\leq a + \frac{2\sigma^2}{na}+ b + \frac{4M}{3n}
\\
&=2\sigma \Big(\frac{\log (2N)}n\Big)^{1/2} +\frac{\sigma}{(n\log (2N))^{1/2}}+ \frac{4M}{3n}\log (2N) + \frac{4M}{3n}
\\
&=\sigma \Big(\frac{ \log (2N)}n\Big)^{1/2} \Big(2+ \frac 1{\log (2N)}\Big) +\frac{4M}{n} \frac{1+  \log(2N)}3
\\
&\le 4\sigma \Big(\frac{\log (2N)}n\Big)^{1/2} + \frac{4M}n \log (2N),
\end{align*}
where we used that $2\log(2N) \geq \log 4\geq 1$.
The proof of \eqref{R1} is complete.
\end{proof}

\begin{remark}
Note that by assumption \eqref{R0} it follows that for any $\phi \in \cF$
\begin{equation}\label{R2}
|\phi(X) -\bE(\phi(X))| \leq M
\quad\text{and}\quad
\bE|\phi(X) -\bE(\phi(X))|^2 \le \sigma^2.
\end{equation}
\end{remark}

\subsection{Proof of Theorem~\ref{Tresh}}

We will carry out this proof in several steps.

First, assuming that $f\in \bL^p$, $1\le p\le\infty$,
we use \eqref{frame-decom}-\eqref{frame2} and \eqref{WDE} to write
\begin{equation*}
\hat{f}_n- f = \sum_{0\le j\le J_n}\sum_{\xi\in\cX_j}
\big(\hat{\beta}_{j\xi}\ONE_{\{\hat{\beta}_{j\xi}>2\lambda_n\}} - \beta_{j\xi}\big)\psi_{j\xi}+\sum_{j>J_n}\Psi_j(\sqrt{L})f,
\end{equation*}
which implies an estimate on the risk as a sum of stochastic and bias terms:
\begin{equation*}
\bE \|\hat{f}_n -f\|_p
\le \sum_{0\le j\le J_n}
\bE\Big\| \sum_{\xi \in \cX_j}(\hat{\beta}_{j, \xi}^{*}-\beta_{j\xi})\psi_{j, \xi} \Big\|_p
+ \Big\|\sum_{j>J_n} \Psi_j(\sqrt L) f \Big\|_p.
\end{equation*}

\subsubsection{Estimation of the bias term}

By the triangle inequality we get
\begin{equation*}
\Big\|\sum_{j > J_n} \Psi_j(\sqrt L) f\Big\|_p \leq \sum_{j> J_n} \|\Psi_j(\sqrt L) f\|_p.
\end{equation*}
Two cases are to be considered here.
Let $r\le p \le \infty$
and set $s_1:=s- d(\frac 1r-\frac 1p)>0$.
Then using Proposition \ref{Pr:Embed} (ii), \eqref{useful-Besov}, and \eqref{bJ}
we obtain
\begin{align*}
\sum_{j> J_n} \|\Psi_j(\sqrt L) f\|_p
&\le \|f\|_{B^{s_1}_{p\tau}}\sum_{j > J_n}b^{-js_1}
\\
&\le c\| f \|_{B^s_{r\tau}} b^{-J_ns_1}
\le c\| f \|_{B^s_{r\tau}} \Big(\frac{\log n}n\Big)^{\frac sd -(\frac 1r -\frac 1p)}.
\end{align*}
Let $p < r \le \infty$.
By H\"{o}lder's inequality (using $\mu(M)<\infty$) and \eqref{useful-Besov} we get
\begin{equation*}
\sum_{j > J_n} \|\Psi_j(\sqrt L) f\|_p
\le c\sum_{j > J_n} \|\Psi_j(\sqrt L) f\|_r
\le c\| f \|_{B^s_{r\tau}}\sum_{j\ge J_n}b^{-js}
\le c\| f \|_{B^s_{r\tau}} \Big(\frac{\log n}n\Big)^{\frac sd}.
\end{equation*}
Therefore, we have the following estimate for the bias
\begin{equation}\label{bias}
\Big\| \sum_{j > J_n} \Psi_j(\sqrt L) f \Big\|_p
\le c\| f \|_{B^s_{r\tau}} \Big(\frac{\log n}n\Big)^{\frac sd -(\frac 1r -\frac 1p)_+}.
\end{equation}

We next show that the rate from above is negligible compared to the rates in \eqref{main-1}-\eqref{main-3}.
First, if $p=\infty$, we have to verify that
$$
\frac sd -\frac 1r > \frac{s-\frac dr}{2(s-d(\frac 1r-\frac 12))}
=\Big(\frac sd -\frac 1r\Big) \frac d{2(s-\frac dr) +d}.
$$
But this is obvious as $s>\frac dr$.

If $s\le \frac{dp}2(\frac 1r-\frac 1p)$, we have $2<p < \infty$ and $r<p$ as $s>\frac dr$.
We have to verify that
$$
\frac sd -\Big(\frac 1r-\frac 1p\Big) > \frac{s-d(\frac 1r-\frac 1p)}{2(s-d(\frac 1r-\frac 12))}
= \Big(\frac sd-\Big(\frac 1r-\frac 1p\Big)\Big)\frac{d}{2(s-\frac dr) +d},
$$
which is obvious.

If $ s > \frac{dp}2(\frac 1r-\frac 1p)$ and $p \leq r$,
we have to show that
$\frac sd > \frac s{2s+d}$,
which again is  obvious.

If $ s > \frac{dp}2(\frac 1r-\frac 1p) >0$ (hence $r<p$)
we have to verify that
$$
\frac sd -\Big(\frac 1r-\frac 1p\Big) > \frac s{2s+d}
\quad \Longleftrightarrow \quad
\frac sd -\frac s{2s+d} >\frac 1r-\frac 1p.
$$
In fact, we have
$$
s > \frac{dp}2\Big(\frac 1r-\frac 1p\Big) \vee \frac dr=:a(p,r).
$$
As the function $s\mapsto \frac sd -\frac s{2s+d}$ is strictly increasing,
we just need to show that
$\frac{a(p,r)}d -\frac{a(p,r)}{2s+d} \geq (\frac 1r-\frac 1p)$.
However,
$$
\frac{dp}2\Big(\frac 1r-\frac 1p\Big)
\le \frac dr
\quad \Longleftrightarrow \quad
p \le r+2
\quad \Longleftrightarrow \quad
\frac{\frac dr}d -\frac{\frac dr}{2s+d} \ge \frac 1r-\frac 1p
$$
and
$$
r+2 <p
\quad \Longleftrightarrow \quad
\frac{\frac{dp}2(\frac 1r-\frac 1p)}d -\frac{\frac{dp}2(\frac 1r-\frac 1p)}{2s+d}
> \frac 1r-\frac 1p,
$$
as it can be easily verified.

From above it follows that the rate in \eqref{bias} is faster than the rates in \eqref{main-1}-\eqref{main-3}.

\subsubsection{Evaluation of the stochastic term}

Note that by \eqref{bg3} we have
\begin{equation}\label{stochast-1}
\bE\Big\| \sum_{\xi \in \cX_j}(\hat{\beta}_{j \xi}^{*}-\beta_{j\xi})\psi_{j\xi} \Big\|_p
\le  c b^{jd(\frac{1}{2}-\frac{1}{p})}
\bE\Big[\Big(\sum_{\xi \in \cX_j}|\hat{\beta}_{j\xi}^{*}-\beta_{j\xi}|^p\Big)^{1/p}\Big],
\quad 1\le p<\infty,
\end{equation}
and this holds with the usual modification when $p=\infty$.

To estimate the stochastic term we will use the following representation
\begin{align*}
\hat{\beta}_{j\xi}^{*}-\beta_{j\xi}
&=(\hat{\beta}_{j\xi}-\beta_{j\xi})\ONE_{\{|\hat{\beta}_{j\xi}| >2\lambda_n\}}
-\beta_{j\xi}\ONE_{\{|\hat{\beta}_{j\xi}|\le 2\lambda_n\}}
\\
&=(\hat{\beta}_{j\xi}-\beta_{j\xi})\ONE_{\{|\hat{\beta}_{j\xi}| >2\lambda_n\}}\ONE_{\{|\beta_{j\xi}|\le\lambda_n\}}
\\
&+(\hat{\beta}_{j\xi}-\beta_{j\xi})\ONE_{\{|\hat{\beta}_{j\xi}| > 2\lambda_n\}}\ONE_{\{|\beta_{j\xi}|>\lambda_n\}}
\\
&-\beta_{j\xi}\ONE_{\{|\hat{\beta}_{j \xi}| \le 2\lambda_n\}}\ONE_{\{|\beta_{j \xi}| \leq 3\lambda_n\}}
\\
&- \beta_{j\xi}\ONE_{\{|\hat{\beta}_{j \xi}| \le 2\lambda_n\}}\ONE_{\{|\beta_{j \xi}| > 3\lambda_n\}}.
\end{align*}
In the case when $1\le p<\infty$, we use this and \eqref{stochast-1} to write
\begin{equation}\label{est-stoch-p}
\sum_{0 \le j \le J_n}
\bE\Big\| \sum_{\xi \in \cX_j}\big(\hat{\beta}_{j \xi}\ONE_{\{|\hat{\beta}_{j\xi}|>2\lambda_n\}}-\beta_{j\xi}\big)\psi_{j \xi} \Big\|_p
\le c\big(I +II +III +IV\big),
\end{equation}
where
\begin{align*}
I &:= \sum_{0 \le j \le J_n} b^{-jd(\frac 1p-\frac 12)}
\bE\Big[\Big(\sum_{\xi \in \cX_j}|\hat{\beta}_{j \xi} -\beta_{j\xi}|^p\Big]
\ONE_{\{|\hat{\beta}_{j \xi}- \beta_{j \xi}| >\lambda_n\}}\Big)^{\frac 1p},
\\
II &:=\sum_{0 \le j \le J_n} b^{-jd(\frac 1p-\frac 12)}
\bE\Big[\Big(\sum_{\xi \in \cX_j}|\hat{\beta}_{j \xi} -\beta_{j\xi}|^p \ONE_{\{|\hat{\beta}_{j \xi}| >2\lambda_n\}}
\ONE_{\{|\beta_{j, \xi}| >\lambda_n\}}\Big)^{\frac 1p}\Big],
\\
III&:=\sum_{0 \le j \le J_n} b^{-jd(\frac 1p-\frac 12)}
\bE\Big[\Big(\sum_{\xi \in \cX_j} |\beta_{j\xi}|^p \ONE_{\{|\hat{\beta}_{j \xi}| \le 2\lambda_n\}}
\ONE_{\{|\beta_{j \xi}| \le 3\lambda_n\}}\Big)^{\frac 1p}\Big]
\\
IV &:=\sum_{0 \le j \le J_n} b^{-jd(\frac 1p-\frac 12)}
\bE\Big[\Big(\sum_{\xi \in \cX_j} |\beta_{j\xi}|^p
\ONE_{\{|\hat{\beta}_{j \xi}- \beta_{j \xi}| >\lambda_n\}}\Big)^{\frac 1p}\Big].
\end{align*}
In the case $p=\infty$ we have
\begin{equation}\label{est-stoch-infty}
\sum_{0 \le j \le J_n}
\bE\Big\| \sum_{\xi \in \cX_j}\big(\hat{\beta}_{j \xi}\ONE_{\{|\hat{\beta}_{j\xi}|>2\lambda_n\}}-\beta_{j\xi}\big)\psi_{j \xi} \Big\|_\infty
\le c\big(I' +II' +III' +IV'\big),
\end{equation}
where
\begin{align*}
I'&:=\sum_{0 \le j \le J_n} b^{j\frac d2}
\bE\Big[\sup_{\xi \in \cX_j}(|\hat{\beta}_{j \xi} -\beta_{j\xi}|
\ONE_{\{|\hat{\beta}_{j \xi}- \beta_{j \xi}| >\lambda_n\}}\Big],
\\
II'&:= \sum_{0 \le j \le J_n} b^{j\frac d2}
\bE\Big[\sup_{\xi \in \cX_j}|\hat{\beta}_{j \xi} -\beta_{j\xi}|
\ONE_{\{|\hat{\beta}_{j \xi}| >2\lambda_n\}}\ONE_{\{|\beta_{j \xi}| >\lambda_n\}}\Big],
\\
III' &:=\sum_{0 \le j \le J_n} b^{j\frac d2}
\bE\Big[\sup_{\xi \in \cX_j} |\beta_{j\xi}|\ONE_{\{|\hat{\beta}_{j\xi}| \le 2\lambda_n\}}
\ONE_{\{|\beta_{j \xi}| \le 3\lambda_n\}}\Big],
\\
IV'&:=\sum_{0 \le j \le J_n} b^{j\frac d2}
\bE\Big[\sup_{\xi \in \cX_j}|\beta_{j\xi}|
\ONE_{\{|\hat{\beta}_{j \xi}- \beta_{j \xi}| >\lambda_n\}}\Big].
\end{align*}

\noindent
\textbf{Estimation of \boldmath $I$ and $I'$.}
As $\card (\cX_j) \le  cb^{jd}$ by (\ref{cardXj}), we derive
\begin{align*}
I&\le  \sum_{0 \le j \le J_n} b^{-jd(\frac 1p-\frac 12)}(\card (\cX_j))^{\frac 1p}
\bE\Big[\sup_{\xi \in \cX_j}|\hat{\beta}_{j \xi} -\beta_{j\xi}|\ONE_{\{|\hat{\beta}_{j \xi}- \beta_{j\xi}| >\lambda_n\}}\Big]
\\
&\le c\sum_{0 \le j \le J_n} b^{j\frac d2}
\bE\Big[\sup_{\xi \in \cX_j}|\hat{\beta}_{j \xi} -\beta_{j\xi}|\ONE_{\{|\hat{\beta}_{j \xi}- \beta_{j \xi}| >\lambda_n\}}\Big]
= cI'.
\end{align*}
Now, in light of Lemma \ref{technical1} we obtain
\begin{align*}
I&\le cI'= c\sum_{0 \le j \le J_n} b^{j\frac d2}
\bE\Big[\sup_{\xi \in \cX_j}( |\hat{\beta}_{j \xi} -\beta_{j\xi}|\ONE_{\{|\hat{\beta}_{j \xi}- \beta_{j \xi}| >\lambda_n\}}\Big]
\\
&\le c\sum_{0 \le j \le J_n} b^{j\frac d2}
 \sum_{\xi \in \cX_j}\bE\big[|\hat{\beta}_{j \xi} -\beta_{j\xi}|\ONE_{\{|\hat{\beta}_{j \xi}- \beta_{j \xi}| >\lambda_n\}}\big]
\\
&= c\sum_{0 \le j \le J_n} b^{j\frac d2} \sum_{\xi \in \cX_j}
\int_0^\infty \bP\big(|\hat{\beta}_{j\xi}- \beta_{j\xi}|\ONE_{\{|\hat{\beta}_{j\xi}- \beta_{j\xi}|>\lambda_n\} }>\lambda\big)d\lambda
\\
&\le c\Big(\frac{\log n}{n}\Big)^{1/2}\frac{1}{n^2} \sum_{0 \le j \le J_n} b^{3j\frac d2}
\le c\Big(\frac{\log n}{n}\Big)^{1/2}\frac{1}{n^2} b^{3J_n\frac d2}
=\frac{c}{n\log n}.
\end{align*}
Therefore,
\begin{equation}\label{est-I}
I\le cI'\le \frac{c}{n\log n},
\end{equation}
and hence the terms $I$ and $I'$ are negligible
compared to the rates in \eqref{main-1}-\eqref{main-3}.

\medskip

\noindent
\textbf{Estimation of \boldmath $IV$ and $IV'$.}
Observe that
\begin{align*}
IV'&=\sum_{0 \le j\le J_n} b^{j\frac d2}
\bE\big[\sup_{\xi \in \cX_j} |\beta_{j\xi}|
\ONE_{\{|\hat{\beta}_{j \xi}- \beta_{j, \xi}| >\lambda_n\}}\big]
\\
&\le \sum_{0 \le j \le J_n} b^{j\frac d2} \sup_{\xi \in \cX_j}  |\beta_{j\xi}|
\sum_{\xi \in \cX_j} \bP( |\hat{ \beta}_{j\xi}-\beta_{j\xi}| >   \lambda_n)=:IV''.
\end{align*}
On the other hand, Jensen's inequality and the fact that $\card (\cX_j) \le  cb^{jd}$ 
yield
\begin{align*}
IV&\le \sum_{0 \leq j \leq J_n} b^{j \frac d2  } b^{-j \frac dp  }
\Big(\bE\Big[\sum_{\xi \in \cX_j} |\beta_{j\xi}|^p \ONE_{\{|\hat{\beta}_{j, \xi}- \beta_{j, \xi}| >\lambda_n\}}\Big]\Big)^{\frac 1p}
\\
&\leq\sum_{0 \le j \le J_n} b^{j \frac d2  } b^{-j \frac dp  }
\bE\big[\sup_{\xi \in \cX_j} |\beta_{j\xi}| \ONE_{\{|\hat{\beta}_{j, \xi}- \beta_{j, \xi}| >\lambda_n\}}\big]
(\card(\cX_j)^{\frac 1p}
\\
&\le c  \sum_{0 \le j \le J_n} b^{j \frac d2  }
\bE\big[\sup_{\xi \in \cX_j} |\beta_{j\xi}| \ONE_{\{|\hat{\beta}_{j, \xi}- \beta_{j, \xi}| >\lambda_n\}}\big]
\le c IV''.
\end{align*}
Now, as $\bP (|\hat{\beta}_{j, \xi}- \beta_{j, \xi}| >\lambda_n) \leq 2n^{-2}$, $0\le j\le J_n$,
by \eqref{Bern4},
and
\begin{equation}\label{est-bet-B}
|\beta_{j\xi}| \le cb^{-j(s+d(\frac 12 - \frac 1r))}\|f\|_{B^s_{r\tau}}, \quad \xi\in\cX_j
\end{equation}
by \eqref{bg4}, we derive
\begin{align*}
IV''&\le cn^{-2}\sum_{0 \le j \le J_n} b^{jd/2}
\sup_{\xi \in \cX_j}  |\beta_{j\xi}|\card(\cX_j)
\\
&\le cn^{-2} \|f\|_{B^s_{r\tau}}
\sum_{0\le j\le J_n}  b^{j\frac{3d}{2}}  b^{-js}b^{jd(\frac 1r-\frac 12)}
\\
&\le cn^{-2} \|f\|_{B^s_{r\tau}}b^{J_n d}\sum_{j=0}^{\infty}  b^{-j(s-\frac{d}{r})}
\\
&\le cn^{-2} \|f\|_{B^s_{r\tau}}\frac{n}{\log n}
=c\| f\|_{B^s_{r\tau}} \frac 1{n\log n}.
\end{align*}
Above we also used the fact that $s>d/r$.
Therefore,
\begin{equation}\label{est-IV}
IV, IV'\le \frac{c}{n\log n},
\end{equation}
and hence the terms $IV$ and $IV'$ are also negligible.

\medskip

\noindent
\textbf{Estimation of \boldmath $II'$ and $III'$.}
%
We first estimate $III'$.
Using \eqref{est-bet-B} we get
\begin{align*}
III'&\le\sum_{0 \le j \le J_n}
\sup_{\xi \in \cX_j}  (b^{j\frac d2}|\beta_{j\xi}|) \wedge ( 3 b^{j\frac d2}\lambda_n )
\\
&\le c\sum_{0 \le j \le J_n}(b^{-j(s-\frac dr)}\|f\|_{B^s_{r\tau}})\wedge( 3 b^{j\frac d2}\lambda_n).
\end{align*}
We now introduce a new parameter $J'_n$ by the identity:
$$
b^{-J'_n (s-\frac dr)} \|f\|_{B^s_{r\tau}}= b^{J'_n \frac d2 }\lambda_n.
$$
Hence,
$$
b^{J'_n (s-  d (\frac 1r-\frac 12))}= \frac{\| f \|_{B^s_{r\tau}}}{\lambda_n}
= \frac{\| f \|_{B^s_{r\tau}}}{\kappa}\Big(\frac n{\log n}\Big)^{\frac 12}.
$$
Then we obtain, taking into account that $s>d/r$,
\begin{align*}
III'&\le c\lambda_n \sum_{0 \leq j <J'_n}
b^{j\frac d2} +  c\| f \|_{B^s_{r\tau}}\sum_{J'_n \le j <J_n} b^{-j (s-\frac dr)}
\\
&\le c \lambda_n b^{J'_n\frac d2}
+ c(s,r) \| f \|_{B^s_{r\tau}}b^{-J'_n (s-\frac dr)}
\\
&\le c\| f \|_{B^s_{r\tau}}b^{-J'_n (s-\frac dr)}
\le c\| f \|_{B^s_{r\tau}}^{\frac{\frac d2}{s-d(\frac 1r-\frac 12)}}
\Big(\kappa \sqrt{\frac{\log n}n}\Big)^{\frac{s-\frac dr}{s-d(\frac 1r-\frac 12)}}
\\
&\le c(s,r, \kappa)  \| f \|_{B^s_{r\tau}}^{\frac{d}{ 2(s- d(\frac 1r -\frac 12)}}
\Big(\frac{\log n}n\Big)^{\frac{s-\frac dr}{ 2(s- d(\frac 1r -\frac 12))}}.
\end{align*}
Therefore,
\begin{equation}\label{est-III'}
III' \le c\Big(\frac{\log n}n\Big)^{\frac{s-\frac dr}{ 2(s- d(\frac 1r -\frac 12))}}.
\end{equation}

To estimate $II'$ we first observe that
\begin{align*}
II'&= \sum_{0 \le j \le J_n} b^{j\frac d2}
\bE\big[\sup_{\xi \in \cX_j}|\hat{\beta}_{j\xi} -\beta_{j\xi}|
\ONE_{\{|\hat{\beta}_{j\xi}| >2\lambda_n\}}\ONE_{\{|\beta_{j, \xi}| >\lambda_n\}}\big]
\\
&\le \sum_{0 \leq j \le J_n} b^{j\frac d2}
\bE\big[\sup_{\xi \in \cX_j}|\hat{\beta}_{j\xi} -\beta_{j\xi}| \ONE_{\{|\beta_{j\xi}| >\lambda_n\}}\big].
\end{align*}
By \eqref{est-bet-B} we have
$$
\kappa \Big(\frac{\log n}n\Big)^{\frac{1}{2}}
=\lambda_n < |\beta_{j\xi}| \le c\| f \|_{B^s_{r\tau}} b^{-j(s-d(\frac 1r-\frac 12))},
\quad 0\le j\le J_n.
$$
Therefore, necessarily
$0\le j \le \tilde{J}_n$,
where
$$
b^{\tilde{J}_n}\sim\Big(\frac{ \| f \|_{B^s_{r\tau}} }{\kappa}\Big)^{\frac 1{s-d( \frac 1r-\frac 12)}  }
\Big(\frac n{\log n}\Big)^{\frac 1{ 2(s-d( \frac 1r-\frac 12))}}
$$
and consequently
$$
II'\le \sum_{0 \leq j <\tilde{J}_n} b^{j\frac d2}
\bE\big[\sup_{\xi \in \cX_j}|\hat{\beta}_{j, \xi} -\beta_{j\xi}|\big].
$$
We next utilize Lemma~\ref{L1}.
Consider the family
$\cF:=\{\tilde{\psi}_{j\xi}: \xi\in\cX_j\}$ with $\card(\cF)\le cb^{jd}$
and by (\ref{bg2})
\begin{align*}
\|\tilde{\psi}_{j\xi}\|_{\infty}\le \cd b^{jd/2}=: M/2
\quad\hbox{and}\quad
\bE|\tilde{\psi}_{j\xi}(X)|^2\le \cd \|f\|_{\infty}=:\sigma^2.
\end{align*}
Then applying Lemma~\ref{L1} we get
\begin{align*}
\bE\big[\sup_{\xi \in \cX_j}|\hat{\beta}_{j, \xi} -\beta_{j\xi}|\big]
&\le c\|f\|_{\infty}^{\frac 12}\Big(\frac{\log (b^{jd})}n\Big)^{\frac 12} + c b^{j\frac d2}\frac{ \log (b^{jd})}{n}
\\
&\le c\|f\|_{\infty}^{\frac 12}\Big(\frac{j}n\Big)^{\frac 12} + cb^{j\frac d2}\frac{j}n.
\end{align*}
Since also $\tilde{J}_n \le c(s,r) \log n$, we get
$$
II'\le c\sum_{0 \le j \le \tilde{J}_n} b^{j\frac d2}
\Big(\|f\|_{\infty}^{\frac 12}\Big(\frac{j}n\Big)^{\frac 12} + b^{j  \frac d2} \frac{j}n \Big)
\le c\|f\|_{\infty}^{\frac 12} \Big(b^{d\tilde{J}_n}\frac{ \tilde{J}_n}n\Big)^{\frac 12}
+ cb^{d\tilde{J}_n}\frac{\tilde{J}_n}n .
$$
But
\begin{align*}
b^{d\tilde{J}_n}  \frac{ \tilde{J}_n}n
&\le c
\| f \|_{B^s_{r\tau}}^{\frac d{s-d( \frac 1r-\frac 12)}}
\Big(\frac n{\log n}\Big)^{\frac d{ 2(s-d( \frac 1r-\frac 12))}}\frac{\log n}n
\\
&= c \| f \|_{B^s_{r\tau}}^{\frac d{s-d( \frac 1r-\frac 12)}}
\Big(\frac{\log n}n\Big)^{\frac{(s-\frac dr)}{ (s-d( \frac 1r-\frac 12))}}.
\end{align*}
Thus, because $s>d/r$ we conclude, using Proposition~\ref{Pr:Embed} (ii),
$$
II'\le c\Big(\|f\|_{B^s_{r\tau}}^{1/2}\|f\|_{B^s_{r\tau}}^{\frac d{2(s-d(\frac 1r-\frac 12))}}
+\| f \|_{B^s_{r\tau}}^{\frac d{s-d( \frac 1r-\frac 12)}}\Big)
\Big(\frac{\log n}n\Big)^{\frac{s-\frac dr}{2(s-d(\frac 1r-\frac 12))}}.
$$
Therefore,
\begin{equation}\label{est-II'}
II' \le c\Big(\frac{\log n}n\Big)^{\frac{s-\frac dr}{2(s-d(\frac 1r -\frac 12))}}.
\end{equation}

\medskip

\noindent
\textbf{Estimation of II and III.} 
As $p\ge1$ Jensen's inequality implies
\begin{align*}
II&\le \sum_{0\le j\le J_n} b^{jd(\frac 12-\frac 1p)}
\bE \Big[\Big( \sum_{\xi \in \cX_j}|\hat{\beta}_{j\xi} -\beta_{j\xi} |^p
 \ONE_{\{ |\beta_{j\xi}| > \lambda_n\}}\Big)^{\frac 1p}\Big]
\\
&\le \sum_{0\le j\le J_n} b^{jd(\frac 12-\frac 1p)}
\Big(\bE\Big[\sum_{\xi \in \cX_j}|\hat{\beta}_{j\xi} -\beta_{j\xi} |^p
\ONE_{\{|\beta_{j\xi}| > \lambda_n\}}\Big]\Big)^{\frac 1p}
\\
&=\sum_{0\le j\le J_n}
\Big(b^{jd(\frac p2-1)} \sum_{\xi \in \cX_j } \ONE_{\{|\beta_{j\xi}| > \lambda_n\}}
\bE\Big[ |\hat{\beta}_{j\xi} -\beta_{j\xi} |^p\Big]\Big)^{\frac 1p}.
\end{align*}
By \eqref{Rosen3} and \eqref{Rosen2} it follows that
$\bE [|\hat{\beta}_{j,k} -\beta_{j,k} |^p] \le c\left(\frac An\right)^{p/2}$
and hence
$$
II \le c\Big(\frac{A}{n}\Big)^{\frac 12} \sum_{0\le j\le J_n} \Big( b^{jd(\frac p2-1)}
\sum_{\xi \in \cX_j} \ONE_{\{ |\beta_{j\xi}| > \lambda_n\}}\Big)^{\frac 1p}.
$$
Let $q<p$. Using the Bienayme-Chebyshev inequality we get
\begin{align}\label{est-II}
II &\le c\Big(\frac{A}{n}\Big)^{\frac 12} \sum_{0\le j\le J_n}
\Big(b^{jd(\frac p2-1)} \frac 1{ \lambda_n^{q}  } \sum_{\xi \in \cX_j}|\beta_{j\xi}|^q\Big) ^{\frac 1p}
\\
&=c(\log n)^{-1/2}\lambda_n^{1-\frac qp}\sum_{0\le j\le J_n}\Big(b^{jd(\frac p2-1)}
\sum_{\xi \in \cX_j}|\beta_{j\xi}|^q \Big)^{\frac 1p}.\nonumber
\end{align}
For the term $III$ we have
\begin{align*}
III &\le \sum_{0\le j\le J_n} b^{jd(\frac 12-\frac 1p)}
\bE\Big[\Big(\sum_{\xi \in \cX_j}|\beta_{j\xi}|^p
\ONE_{\{|\beta_{j\xi}| \leq 3 \lambda_n\}}\Big)^{\frac 1p}\Big]
\\
&=\sum_{0\le j\le J_n}\Big(  b^{jd(\frac p2-1)} \sum_{\xi \in \cX_j}|\beta_{j\xi}|^p \ONE_{\{|\beta_{j\xi}|
\le 3 \lambda_n\}}\Big)^{\frac 1p}.
\end{align*}
Assuming that $\lambda>0$, $0<q < p$, and $\sigma$ is a positive measure on a measure space $X$, we have
\begin{align*}
\int_X \ONE_{\{|f| \leq \lambda\}} |f|^p d\sigma
&\le \int_X (|f|\wedge \lambda)^p d\sigma
= \int_0^\lambda px^{p-1} \sigma (|f|>x)dx
\\
&\le \int_0^\lambda px^{p-1} \frac{\| f \|_q^q}{x^q}dx
= \frac p{p-q}\| f \|_q^q \lambda^{p-q}.
\end{align*}
Therefore, if $q<p$ we have
\begin{align*}
III &\le \sum_{0\le j\le J_n}\Big( b^{jd(\frac p2-1)} \sum_{\xi \in \cX_j}|\beta_{j\xi}|^p
\ONE_{\{|\beta_{j\xi}|\leq 3 \lambda_n\}}\Big) ^{\frac 1p}
\\
&\le c(p,q)\sum_{0\le j\le J_n}\Big(\lambda_n^{p-q} b^{jd(\frac p2-1)}
\sum_{\xi \in \cX_j}|\beta_{j\xi}|^q\Big)^{\frac 1p}
\\
&\le c\lambda_n^{1-\frac qp}\sum_{0\le j\le J_n}\Big(b^{jd(\frac p2-1)}
\sum_{\xi \in \cX_j}|\beta_{j\xi}|^q\Big)^{\frac 1p}.
\end{align*}
This coupled with \eqref{est-II} yields
\begin{equation}\label{est-II-III}
II+III \le c\lambda_n^{1-\frac qp}\sum_{0\le j\le J_n}\Big(b^{jd(\frac p2-1)}
\sum_{\xi \in \cX_j}|\beta_{j\xi}|^q\Big)^{\frac 1p},
\quad 0<q<p.
\end{equation}
Assuming $0<q<p$ we set $s':=\frac{d(p-q)}{2q}$.
By \eqref{bg4} it follows that
\begin{equation*}
\sum_{\xi\in\cX_j}|\beta_{j\xi}(f)|^q
\le c b^{-jq(s'+d(\frac{1}{2}-\frac{1}{q}))}\|f\|_{B^{s'}_{q\tau}}^q,
\quad j\ge 0
\end{equation*}
We combine this with \eqref{est-II-III} and the fact that $J_n\le c\log n$ to obtain
\begin{align}\label{II,III}
II+ III
&\le c\lambda_n^{1-\frac qp}\|f\|_{B^{s'}_{q\tau}}^{\frac qp}
\sum_{0\le j\le J_n}\Big(b^{jd(\frac p2-1)}b^{-jq(s'+d(\frac{1}{2}-\frac{1}{q}))}\Big)^{\frac 1p}
\\
&\le c\lambda_n^{1-\frac qp}\|f\|_{B^{s'}_{q\tau}}^{\frac qp}J_n
\le c \log n\Big(\kappa\frac{\log n}n\Big)^{\frac{p-q}{2p}}\|f\|_{B^{s'}_{q\tau }}^{\frac qp}.\nonumber
\end{align}
Thus, in what follows we have to show that
$\|f\|_{B^{s'}_{q\tau}}\le c\|f\|_{B^{s}_{r\tau}}$
for a suitable $q <p$.

\subsubsection{Proof of Theorem~\ref{Tresh} $(i)$}

Let $1\le r\le\infty$, $0<\tau\le\infty$, $s>d/r$, and $m>0$.
Assume $f\in B^s_{r\tau}$ and $\|f\|_{B^s_{r\tau}}\le m$.
Then combining \eqref{est-stoch-infty}, \eqref{est-I}, \eqref{est-IV}, \eqref{est-III'}, \eqref{est-II'}, and \eqref{bias}
we obtain
$$
\bE\|\hat{f}_n-f\|_\infty \le c(m)\Big(\frac{\log n}n\Big)^{\frac{s-\frac dr}{2(s-d(\frac 1r -\frac 12))}},
$$
which confirms \eqref{main-1}.

\subsubsection{Proof of Theorem~\ref{Tresh} $(ii)$}

Let $1\le r\le\infty$, $0<\tau\le\infty$, $1\le p<\infty$, $s>d/r$, and $m>0$.
Assume $f\in B^s_{r\tau}$, $s\ge\frac{dp}{2}\big(\frac{1}{r}-\frac{1}{p}\big)$,
and $\|f\|_{B^s_{r\tau}}\le m$.
Choose $s'=s$.
Then using that from above
$$
\frac{d(p-q)}{2q}=s'=s\ge\frac{dp}{2}\big(\frac{1}{r}-\frac{1}{p}\big)
$$
we obtain
$$
q=\frac{dp}{2s+d}<p,
\quad \frac{p-q}{2p}=\frac{s}{2s+d}
\quad\text{and}\quad q\le r.
$$
From Proposition~\ref{Pr:Embed} (i) and (iii) it follows that
\begin{equation*}
\|f\|_{B^{s'}_{q\tau}} = \|f\|_{B^s_{q\tau}} \le c\|f\|_{B^s_{r\tau}}\le cm
\quad\hbox{and}\quad
\|f\|_\infty \le c\|f\|_{B^s_{r\tau}}\le cm.
\end{equation*}
and from \eqref{II,III} we conclude
\begin{align*}
II+III\le c\log n\left(\kappa\frac{\log n}n\right)^{\frac{s}{2s+d}}\|f \|_{B^{s}_{r\tau }}^{\frac{d}{2s+d}}
\le c(m)\log n\left(\frac{\log n}n\right)^{\frac{s}{2s+d}}.
\end{align*}
This estimate along with \eqref{est-stoch-p}, \eqref{est-I}, \eqref{est-IV}, and \eqref{bias}
imply \eqref{main-2}.

\subsubsection{Proof of Theorem~\ref{Tresh} $(iii)$}

Under the hypotheses of Theorem~\ref{Tresh}
let $f\in B^s_{r\tau}$, $s<\frac{dp}{2}\big(\frac{1}{r}-\frac{1}{p}\big)$,
and $\|f\|_{B^s_{r\tau}}\le m$.
In this case we choose $s':=s-d\big(\frac{1}{r}-\frac{1}{p}\big)>0$.
As above $s'=\frac{d(p-q)}{2q}$ and hence
$$
q=\frac{dp\big(\frac{1}{2}-\frac{1}{p}\big)}{s-d\big(\frac{1}{r}-\frac{1}{2}\big)}\ge r,
\quad p>2,\quad\text{and}\quad
\frac{p-q}{2p}=\frac{s-d\big(\frac{1}{r}-\frac{1}{p}\big)}{2\big(s-d\big(\frac{1}{r}-\frac{1}{2}\big)\big)}.
$$
By Proposition \ref{Pr:Embed} (ii) it follows that
$\|f\|_{B^{s'}_{q\tau}}\le c\|f\|_{B^s_{r\tau}}$
and hence from \eqref{II,III}
$$
II+III\le c(m) \log n\Big(\frac{\log n}n\Big)^{\frac{s-d(\frac{1}{r}-\frac{1}{p})}{2(s-d(\frac{1}{r}-\frac{1}{2}))}}.
$$
Combining this with \eqref{est-stoch-p}, \eqref{est-I}, \eqref{est-IV}, and \eqref{bias} imply \eqref{main-3}.

The proof of Theorem~\ref{Tresh} is complete.
\qed

\section{Examples of settings covered by our theory}\label{Examples}

In this section we present a number of examples of settings
that are covered by the general framework from Section~\ref{sec:setting}.

Clearly, the classical setup of $\MM=\R^d$ equipped with the Euclidean distance and Lebesgue measure,
and with $L=-\Delta$ the Laplacian obeys conditions {\bf C1}--{\bf C5}, together with {\bf C1A} from Section~\ref{sec:setting}.
Therefore, our results on density estimators apply.
They are compatible with the existing upper bound results.
For instance, our result on adaptive wavelet estimators (Theorem~\ref{Tresh})
is compatible with the result in the classical setting on $\R^d$ from \cite{DJ}.

The unit sphere $\bS^d$ in $\R^{d+1}$ equipped with the standard (geodesic) distance, measure, and
$-L$ being the Laplace-Beltrami operator is another example
of a~setup that obeys conditions {\bf C1}--{\bf C5}, together with {\bf C1A}.
Our upper bound result on adaptive wavelet estimators (Theorem~\ref{Tresh})
is compatible with the upper bound estimate on the adaptive needlet estimator in \cite{BKMP}.

A natural generalization of the above setup is the case of a compact Riemannian manifold $\MM\subset\R^m$
equipped with the natural Riemannian measure, geodesic distance,
and $-L$ being the Laplace-Beltrami operator.
Then conditions {\bf C1}--{\bf C5}, together with {\bf C1A} from Section~\ref{sec:setting} are satisfied.

For other examples on Riemannian manifolds and Lie groups we refer the reader to
\cite{CKP} and the references therein.

We next describe the recently developed in \cite{GPY} general setting
of a subset of $\R^n$ complemented by a differential operator $L$
that is a realization in local coordinates of a weighted Laplace operator on suitable subset of a Riemannian manifold.
In~this setting conditions {\bf C1}--{\bf C5} and in some cases {\bf C1A} from Section~\ref{sec:setting} are obeyed.
This setting covers the weighted settings on the interval, ball, and simplex,
which we will describe in more detail as well.

\subsection{Convex subsets of Riemannian manifolds and counterparts on \boldmath $\R^n$}\label{subsec:convex}

We assume that
$V\subset \bR^m$ is a connected open set in $\bR^m$ with the properties:
$X:=\overline{V}$ is compact,
$\mathring{X}=V$, and
$X \setminus V$ is of Lebesgue measure zero.
Let $L$ be a differential operator of the form
\begin{equation}\label{def-L}
L= \sum_{i,j=1}^m a_{ij}(x) \partial_i\partial_j + \sum_{j=1}^m b_j(x)\partial_j,
\end{equation}
where $a_{ij}$ and $b_j$ are polynomials of degrees two and one, respectively.
The underlying space is $\bL^2(V, \mu)$, where
$d\mu(x):=\ww(x) dx$
with
$\ww\in C^\infty(V)$, $\ww>0$,
and $\int_V \ww(x) dx <\infty$.

On the other hand, we assume that there is (a closely related) counterpart to the above setup.
Namely, we assume that $(M, d, \nu)$ is an $m$-dimensional complete Riemannian manifold and $M\subset \bR^m$,
where the Riemannian metric is induced by the inner product on $\bR^m$.
We stipulate two conditions on $(M, d, \nu)$:
(i) the volume doubling condition is valid,
and
(ii) the Poincar\'{e} inequality holds true (see \cite{GPY}).

Further, we assume that $(U, \varphi)$ is a chart on $M$,
where $U$ is a convex open relatively compact subset of $M$ such that
$\varphi$ maps diffeomorphically $U$ onto $V$,
where $V\subset \bR^d$ is from above.
We set $\phi:=\varphi^{-1}$ and $Y:=\overline{U}$.

The {\em key assumption} is that the map $\phi$ provides
an one-to-one correspondence between the elements of the setting
on $X$ from above and the setting on~$Y$.
More precisely,
it is assumed that the distance $\rho(\cdot, \cdot)$ on $X$ is induced by the geodesic distance $d(\cdot, \cdot)$ on $Y\subset M$.
The weighted measure $d\mu(x):=\ww(x) dx$ on $X$ is also induced by the respective wighted measure $\nu_w$ on $Y$.
Namely,
assuming that $g(x)=(g_{ij}(x))$ is the Riemannian tensor
it is stipulated that
$w>0$ is a $C^\infty(U)$ weight function that is
compatible with $\ww$ from above in the following sense:
\begin{equation}\label{weights}
\ww(x) :=  w(\phi(x)) \sqrt{\det g(x)},
\quad x\in V,
\end{equation}
It is assumed that $\nu_w= wd\nu$.
It is also assumed that the operator $L$ from \eqref{def-L} is a~realization in local coordinates
(via the chart $(U, \varphi)$) of the weighted Laplacian
$$
\Delta_w f:= \frac{1}{w}\diver (w\nabla f)
\quad\hbox{on $Y\subset M\quad$ (see \cite{GPY}).}
$$
In addition, it is assumed that the volume doubling condition \eqref{doubl-0} on $Y$
(and hence on $X$) is valid,
and a natural regularity condition on the weighted functions and
suitable Green's theorem are verified.
See \cite{GPY} for the details.

In \cite{GPY} it is shown that under the above conditions
the heat kernels associated to the operator $L$ and weighted Laplacian $\Delta_w$
have Gaussian localization just as in \eqref{Gauss-local} and as a consequence
the H\"{o}lder continuity (see \eqref{lip}) is valid.
Furthermore, these heat kernels have the Markov property (see \eqref{hol3}).

As a result, in the above described general setting conditions {\bf C1}--{\bf C5} in Section~\ref{sec:setting} are obeyed
and our results on kernel and linear wavelet density estimators apply
to the settings on $X\subset \R^m$ and on $Y\subset M$.
Furthermore, if the weight function $w\equiv 1$ on $Y$, then the measures $\mu$ on $X$ and $\nu_w$ on $Y$
verify the Ahlfors condition \eqref{doubling-0}, i.e. condition {\bf C1A} is satisfied,
and consequently our result on adaptive wavelet estimators (Theorem~\ref{Tresh}) applies.

We next show how the general theory described above is implemented in
specific settings on $[-1,1]$, the ball and simplex.

\subsection{Specific examples where our kernel and wavelet estimators work}\label{subsubsec:kern-walelet}

\subsubsection{Jacobi operator on $[-1, 1]$}\label{subsubsec:interval}

We consider the classical setting of $\MM =[-1, 1]$ equipped with the weighted measure
\begin{equation*}
d\mu(x) := w(x) dx = (1-x)^{\alpha}(1+x)^{\beta} dx, \quad \alpha, \beta>-1,
\end{equation*}
and the distance
$\rho(x,y) := |\arccos x - \arccos y|$,
complemented with the classical Jacobi operator, defined by
\begin{equation*}
Lf(x):=-\frac{\big[w(x)(1-x^2)f'(x)\big]'}{w(x)}.
\end{equation*}
As is well known the Jacobi polynomials are eigenfunctions of this operator.
Denote $B(x, r):=\{y\in [-1,1]: \rho(x, y)<r\}$.
It is easy to see that (see e.g. \cite{CKP})
\begin{equation}\label{int-measure}
|B(x, r)|\sim r(1-x+r^2)^{\alpha+1/2}(1+x+r^2)^{\beta+1/2}, \quad x\in [-1, 1], \; 0<r\le \pi.
\end{equation}
Hence, we have a doubling metric measure space with homogeneous dimension $d=1+(2\alpha+1)_+ \vee(2\beta+1)_+$.
More importantly, as is shown in \cite{CKP} the associated heat kernel
has Gaussian localization, H\"{o}lder continuity, and the Markov property (see also \cite{GPY}).
Therefore, conditions {\bf C1}--{\bf C5} in Section~\ref{sec:setting} are obeyed.

\medskip

\noindent
{\bf Ahlfors space on \boldmath $[-1, 1]$.}
In the above setting, if $\alpha=\beta=-1/2$
then from \eqref{int-measure}
\begin{equation*}
|B(x, r)|\sim r, \quad x\in [-1, 1], \; 0<r\le \pi.
\end{equation*}
Therefore, condition {\bf C1A} (see \S\ref{sec:setting}) is obeyed with $d=1$.

\subsubsection{Weighted unit ball}\label{subsubsec:ball}

Consider the case when $\MM$ is
$\bB^m:=\big\{x\in\R^m: \|x\|<1\big\}$ the unit ball in $\R^m$
equipped with the measure
\begin{equation*}\label{def-meas-ball}
d\mu := (1-\| x\|^2)^{\gamma-1/2} dx, \quad \gamma >-1,
\end{equation*}
and the distance
\begin{equation*}
\rho(x,y) := \arccos \big(\langle x, y\rangle + \sqrt{1-\| x\|^2}\sqrt{1-\| y\|^2}\big),
\end{equation*}
where $\langle x, y\rangle$ is the inner product of $x, y\in \R^m$
and $\|x\|:= \sqrt{\langle x, x\rangle}$.
Denoting
$
B(x, r):=\{y\in \bB^m: \rho(x,y)<r\}
$
it is easy to show (see \cite{DaiXu}) that
\begin{equation}\label{measure-ball}
|B(x, r)| \sim r^m(1-\|x\|^2+r^2)^\gamma,
\end{equation}
which implies that $(\MM, \mu, \rho)$
obeys the doubling condition \eqref{doubl-0} and non-collapsing condition \eqref{non-collapsing} and it is of homogeneous dimension $d=m+2\gamma_+$.

Consider the operator
\begin{equation*}
L:= -\sum_{i=1}^m (1-x_i^2)\partial^2_i + 2\sum_{1\le i < j \le m}x_i x_j\partial_i\partial_j
+ (n+2 \gamma)\sum_{i=1}^m x_i \partial_i,
\end{equation*}
acting on sufficiently smooth functions on $\bB^m$.
This operator is essentially self-adjoint and positive (see \cite{DaiXu} and also \cite{GPY}).
More importantly, its heat kernel
has Gaussian localization, H\"{o}lder continuity, and the Markov property (see \cite{GPY}).
In fact, $L$ is a realization of a weighted Laplace operator on the upper hemisphere of $\R^{m+1}$
in local coordinates (see \cite{GPY}).
Consequently, conditions {\bf C1}--{\bf C5} in Section \ref{sec:setting} are verified
and this setting falls in the general framework from Section \ref{sec:setting}.

\medskip

\noindent
{\bf Ahlfors space on the unit ball.}
Assume that in the above setting $\gamma=0$,
i.e. the measure on $\bB^m$ is $d\mu := (1-\| x\|^2)^{-1/2} dx$.
Then from \eqref{measure-ball}
\begin{equation*}
|B(x, r)|\sim r^m,
\quad x\in \bB^m, \; 0<r\le 1,
\end{equation*}
and hence condition {\bf C1A} in Section~\ref{sec:setting} is obeyed with $d=m$.
Thus this is another example of an Ahlfors space.

\subsubsection{Weighted simplex}\label{subsubsec:simplex}

We now consider the simplex
\begin{equation*}
\bT^m:=\Big\{x \in \bR^d: x_1 > 0,\dots, x_m>0,\; |x| < 1 \Big\},
\quad |x|:= x_1+\cdots+x_m,
\end{equation*}
in $\bR^m$ equipped with the measure
\begin{equation*}
d\mu(x) = \prod_{i=1}^{m} x_{i}^{\kappa_i-1/2}(1-|x|)^{\kappa_{m+1}-1/2}dx,
\quad  \kappa_i >-1/2,
\end{equation*}
and distance
\begin{equation*}
\rho(x,y) = \arccos \Big(\sum_{i=1}^m \sqrt{x_i y_i} + \sqrt{1-|x|}\sqrt{1-|y|}\Big).
\end{equation*}
Similarly as before we use the notation:
$B(x, r):=\{y\in\bT^m: \rho(x, y)<r\}.$
It is known (see \cite{GPY}) that
\begin{equation}\label{V-ball-simpl}
|B(x, r)| \sim r^m (1-|x|+r^2)^{\kappa_{m+1}}\prod_{i=1}^m(x_i+r^2)^{\kappa_i}.
\end{equation}
Hence, the doubling condition \eqref{doubl-0} and non-collapsing condition \eqref{non-collapsing} are satisfied. Moreover $d=m+2\big((\kappa_1)_+ +\cdots+(\kappa_{m+1})_+\big)$.

It is natural to consider the operator
\begin{equation*}
L := -\sum_{i=1}^m x_i\partial_i^2 + \sum_{i=1}^m\sum_{j=1}^m x_ix_j \partial_i\partial_j
- \sum_{i=1}^m \big(\kappa_i + \tfrac12 -(|\kappa|+ \tfrac{m+1}{2}) x_i\big) \partial_i
\end{equation*}
with $|\kappa|:=\kappa_1+\dots+\kappa_{m+1}$.
In \cite{GPY}
(see also \cite{GPY2})
 it is shown that
this operator is essentially self-adjoint and positive.
Furthermore, its heat kernel
has Gaussian localization, H\"{o}lder continuity, and the Markov property.
It is important to point out that the operator $L$ is a realization
of a weighted Laplacian on the sphere in the first octant in local coordinates (see \cite{GPY}).
Thus, conditions {\bf C1}--{\bf C5} in \S\ref{sec:setting} are verified
and this setup is covered by the general setting from Section~\ref{sec:setting}.

\medskip

\noindent
{\bf Ahlfors space on the simplex.}
If above $\kappa_i=0$, $i=1, \dots, m+1$,
then the measure is given by
$d\mu(x) = \prod_{i=1}^{m} x_{i}^{-1/2}(1-|x|)^{-1/2}dx$,
and \eqref{V-ball-simpl} yields
\begin{equation*}
|B(x, r)|\sim r^m,
\quad x\in \bT, \; 0<r\le 1.
\end{equation*}
Therefore, condition {\bf C1A} is again satisfied
and this is also an example of an Ahlfors space with $d=m$
where our result on adaptive wavelet estimators (Theorem~\ref{Tresh}) applies.

\section{Appendix}\label{sec:appendix}

\subsection{Proof of Proposition~\ref{prop:nikolski}}

Let $1\le p<q\le\infty$ and $g\in \Sigma_\lambda^p$, $\lambda\ge 1$.

Let $\varphi \in C^\infty_0(\R)$ be an even function with the properties:
$\supp \varphi \subset [-2, 2]$, $0\le \varphi\le 1$, and $\varphi(u)=1$ for $u\in [-1,1]$.

By Theorem~\ref{thm:S-local-kernels} $\varphi(\delta\sqrt{L})$ is an integral operator
whose kernel $\varphi(\delta\sqrt{L})(x, y)$ is real-valued, symmetric, and such that
for any $k>d$, $x, y\in \MM$, and $\delta>0$
\begin{equation}\label{local-ker}
|\varphi(\delta\sqrt{L})(x, y)|
\le c(k)|B(x, \delta)|^{-1}(1+\delta^{-1}\rho(x, y))^{-k},
\end{equation}
where the constant $c(k)>0$ depends only on $k$, $\varphi$, and the constants from the setting in Section~\ref{sec:setting}.

Observe that since $\varphi(u)=1$ for $u\in [-1,1]$ and $g\in \Sigma_\lambda^p$, we have
\begin{equation}\label{reproduce}
\varphi(\delta\sqrt{L})g=g
\quad\hbox{if}\quad \delta = 1/\lambda.
\end{equation}

Set $\delta:=1/\lambda$ and $k:=d+1$.
Define $r>1$ from the identity $1/p-1/q=1-1/r$.
Using \eqref{local-ker}, \eqref{tech-1}, and \eqref{DNC} we obtain
\begin{align}\label{norm-r}
\|\varphi(\delta\sqrt{L})(x, \cdot)\|_r 
&\le c(d+1) |B(x, \delta)|^{-1}\Big(\int_\MM (1+\delta^{-1}\rho(x, y))^{-r(d+1)} d\mu(y)\Big)^{1/r} \notag
\\
&\le c_d \tilde{c}|B(x, \delta)|^{1/r-1}
\le c_d \tilde{c} (c_0/c_3)^{1-1/r}\delta^{d(1/r-1)}
\le c_\star \delta^{d(1/r-1)}.
\end{align}
Here $\tilde{c}:= (2^{-d}-2^{-d-1})^{-1}$,
$c_d=c(d+1)$ is from \eqref{local-ker},
the constants $c_0, c_3>0$ are from \eqref{doubl-0} and \eqref{non-collapsing}.
We may assume that $c_0/c_3\ge 1$ and take
$c_\star = c_d \tilde{c} (c_0/c_3)$,
which is independent of $p$ and $q$.

Thus, using the symmetry we have
\begin{equation*}
\|\varphi(\delta\sqrt{L})(x, \cdot)\|_r
= \|\varphi(\delta\sqrt{L})(\cdot, y)\|_r
\le c_\star \delta^{d(1/r-1)}
= c_\star \delta^{d(1/q-1/p)}
\end{equation*}
We now use the well known Theorem 6.36 from \cite{Folland} to conclude that
\begin{equation*}
\|g\|_q=\|\varphi(\delta\sqrt{L})g\|_q
\le c_\star \delta^{d(1/q-1/p)}\|g\|_p
= c_\star \lambda^{d(1/p-1/q)}\|g\|_p.
\end{equation*}
Here we also used \eqref{reproduce}.
The proof is complete.
\qed

\subsection{Proof of Proposition \ref{Pr:Embed}}

(i) Let $1\le q\le r$, $0<\tau\le\infty$, $s>0$, $\mu(\MM)<\infty$, and $f\in B^s_{r\tau}$.
Let $\Phi_j$, $j\ge 0$, be the functions from the definition of Besov spaces, see Definition~\ref{def:Besov}.
Then by H\"{o}lder's inequality
$\|\Phi_j(\sqrt{L})f\|_q\le\mu(\MM)^{\frac{1}{q}-\frac{1}{r}}\|\Phi_j(\sqrt{L})f\|_r$,
which readily implies
$\|f\|_{B^s_{q\tau}}\le \mu(\MM)^{\frac{1}{q}-\frac{1}{r}} \|f\|_{B^s_{r\tau}}$
as claimed.

\smallskip

(ii) Let $1\le r\le q\le\infty$, $0<\tau\le\infty$, and $s>d\big(\frac{1}{r}-\frac{1}{q}\big)$.
Assume $f\in B^s_{r\tau}$.
Let $\Phi_j$, $j\ge 0$, be the functions from the definition of Besov spaces.
Clearly, $\Phi_j \subset [b^{j-1}, b^{j+1}]$ and hence
$\Phi_j(\sqrt{L})f\in \Sigma_{b^{j+1}}$.
Applying Proposition~\ref{prop:nikolski} we obtain
\begin{equation}\label{nikolski-embed}
\|\Phi_j(\sqrt{L})f\|_q \le c b^{jd(\frac 1r-\frac 1q)}\|\Phi_j(\sqrt{L})f\|_r.
\end{equation}
From this it readily follows that
$\|f\|_{B^{s-d(\frac{1}{r}-\frac{1}{q})}_{q\tau}}\le c\|f\|_{B^s_{r\tau}}$
as claimed.

\smallskip

(iii)
Assume $1\le r\le\infty$, $0<\tau\le\infty$, $s>d/r$, and $f\in B^s_{r\tau}$.
Let the functions $\Phi_j$, $j\ge 0$, in the definition of Besov spaces be selected so that
$\sum_{j=0}^{\infty}\Phi_j(\lambda)=1$ for $\lambda\ge0$.
For example, the functions $\Psi_j$, $j\ge 0$, from the definition of the frames in \S\ref{Frames}
have this property.
Then as in \eqref{frame-decom} we have
$f=\sum_{j=0}^{\infty}\Phi_j(\sqrt{L})f$.
Just as in \eqref{nikolski-embed} by Proposition~\ref{prop:nikolski} we have
\begin{equation*}
\|\Phi_j(\sqrt{L})f\|_\infty \le c b^{jd/r}\|\Phi_j(\sqrt{L})f\|_r.
\end{equation*}
Also, evidently by the definition of Besov spaces
$\|\Phi_j(\sqrt{L})f\|_r \le cb^{-js}\|f\|_{B^s_{r\tau}}$.
We use these two estimates and the fact that $s>d/r$ to obtain
\begin{align*}
\|f\|_{\infty}&\le\sum_{j=0}^{\infty}\|\Phi_j(\sqrt{L})f\|_{\infty}\le c\sum_{j=0}^{\infty}b^{jd/r}\|\Phi_j(\sqrt{L})f\|_r
\\
&\le c\sum_{j=0}^{\infty}b^{-j(s-d/r)}\|f\|_{B^s_{r\tau}}\le c\|f\|_{B^s_{r\tau}}.
\end{align*}

(iv)
Let $1\le p \leq r\le\infty$, $0<\tau\le\infty$, $s>0$, and $\mu(\MM)<\infty$.
Assume $f\in B^s_{r\tau}$.
Let $\Phi_j$, $j\ge 0$, be as in the proof of (iii) above.
Since $p\ge1$ and $s>0$ we have
$$
\|f\|_{p}\le\sum_{j=0}^{\infty}\|\Phi_j(\sqrt{L})f\|_{p}\le c\sum_{j=0}^{\infty}b^{-js}\|f\|_{B^s_{p\tau}}
\le c\|f\|_{B^s_{r\tau}},
$$
where for the last inequality we used (i).
\qed

\subsection{Proof of Proposition~\ref{prop:approx}}

Let $\bPhi_0$ be just as the function $\Phi$ in the definition of Besov spaces
in \S\ref{subsec:Besov} (see \eqref{Phi0}), that is,
$\bPhi_0\in C^{\infty}(\bR_+)$ is a real-valued function with the properties:
$\supp \bPhi_0\subset [0,b]$, $\bPhi_0(\lambda)=1$ for $\lambda\in[0,1]$,
and $\bPhi_0(\lambda)\ge c>0$ for $\lambda\in[0,b^{3/4}]$.
Set
$\bPhi(\lambda):= \bPhi_0(\lambda)-\bPhi_0(b\lambda)$
and
$\bPhi_j(\lambda):=\bPhi(b^{-j}\lambda)$, $j\ge 1$.
Clearly,
$\sum_{j\ge 0} \bPhi_j(\lambda)= 1$, $\lambda\in [0,\infty)$.

Let $0<\delta \le (2b^2)^{-1}$.
The case when $(2b^2)^{-1}<\delta \le 1$ is easier; we omit it.
Choose $\nu\ge 2$ and $\ell\ge 1$ ($\nu,\ell\in \bN$) so that
$b^{\nu}\le 1/(2\delta) < b^{\nu+1}$ and $1/\delta\le b^{\nu+\ell}$
($\ell:=\lfloor \log (2b)/\log b \rfloor +1$ will do).
From above and the properties of $\Phi$ in \eqref{prop-Phi}
it readily follows that
\begin{equation*}
1-\Phi(\delta\lambda)
= (1-\Phi(\delta\lambda))\sum_{j=\nu}^{\nu+\ell} \bPhi(b^{-j}\lambda)
+ \sum_{j=\nu+\ell+1}^\infty \bPhi(b^{-j}\lambda).
\end{equation*}
Set
\begin{equation}\label{def-Lam-T}
\Lambda(\lambda):= (1-\Phi(\delta b^\nu\lambda))\sum_{j=0}^{\ell} \bPhi(b^{-j}\lambda)
\quad\hbox{and}\quad
\Theta(\lambda):= \sum_{j=\nu-1}^{\nu+\ell+2} \bPhi(b^{-j}\lambda).
\end{equation}
By the construction of $\bPhi_j$ it follows that
$\supp \Lambda\subset [1, b^{\ell+1}]$ and
$\Theta(\lambda)=1$ for $\lambda\in[b^\nu, b^{\nu+\ell+1}]$,
implying
$\Lambda(b^{-\nu}\lambda)= \Lambda(b^{-\nu}\lambda)\Theta(\lambda)$.
Therefore, we have
\begin{equation*}
1-\Phi(\delta\lambda)
= \Lambda(b^{-\nu}\lambda)
+ \sum_{j=\nu+\ell+1}^\infty \bPhi(b^{-j}\lambda)
= \Lambda(b^{-\nu}\lambda)\Theta(\lambda)
+ \sum_{j=\nu+\ell+1}^\infty \bPhi(b^{-j}\lambda).
\end{equation*}
From this we infer that for any $f\in \bL^p$
\begin{equation*}
f-\Phi(\delta\sqrt{L})f
= \Lambda(b^{-\nu}\sqrt{L})\Theta(\sqrt{L})f
+ \sum_{j=\nu+\ell+1}^\infty \bPhi(b^{-j}\sqrt{L})f
\end{equation*}
and hence
\begin{equation}\label{est-Lp-norm}
\|f-\Phi(\delta\sqrt{L})f\|_p
\le \|\Lambda(b^{-\nu}\sqrt{L})\Theta(\sqrt{L})f\|_p
+ \sum_{j=\nu+\ell+1}^\infty \|\bPhi(b^{-j}\sqrt{L})f\|_p
\end{equation}
From the definition of $\Lambda(\lambda)$ in \eqref{def-Lam-T} it follow that
$\Lambda\in C^\infty(\bR_+)$, $\supp \Lambda\subset [1, b^{\ell+1}]$,
and $\|\Lambda^{(r)}\|_\infty\le c_r$ for $r=0, 1, \dots$
with the constant $c_r>0$ depending only on $r$, $\Phi$, $\bPhi$, $b$.
Now, we invoke Theorem~\ref{thm:S-local-kernels} to conclude that
$\Lambda(b^{-\nu}\sqrt{L})$ is an integral operator with kernel $\Lambda(b^{-\nu}\sqrt{L})(x, y)$
satisfying
\begin{equation*}
|\Lambda(b^{-\nu}\sqrt{L})(x, y)| \le c(k)b^{d\nu}\big(1+b^\nu\rho(x,y)\big)^k,
\quad x, y\in\MM, \; k>d.
\end{equation*}
As in \eqref{PHI1} this with $k=d+1$ implies
$
\|\Lambda(b^{-\nu}\sqrt{L})(\cdot, y)\|_1
= \|\Lambda(b^{-\nu}\sqrt{L})(x,\cdot)\|_1
\le c.
$
and applying Schur's lemma it follows that
$\|\Lambda(b^{-\nu}\sqrt{L})\|_{p\to p} \le c<\infty$.
As a consequence, we get
$$
\|\Lambda(b^{-\nu}\sqrt{L})\Theta(\sqrt{L})f\|_p
\le c\|\Theta(\sqrt{L})f\|_p
\le c\sum_{j=\nu-1}^{\nu+\ell+2} \|\bPhi(b^{-j}\sqrt{L})f\|_p.
$$
This coupled with \eqref{est-Lp-norm} implies, when $0<q<\infty$,
\begin{align*}
\|f-\Phi(\delta\sqrt{L})f\|_p
&\le c\sum_{j=\nu-1}^\infty \|\bPhi(b^{-j}\sqrt{L})f\|_p
\le c'b^{-s\nu}\Big(\sum_{j=\nu-1}^\infty\big(b^{sj}\|\bPhi(b^{-j}\sqrt{L})f\|_p\big)^q\Big)^{1/q}
\\
&\le c\delta^s \Big(\sum_{j=0}^\infty\big(b^{sj}\|\bPhi(b^{-j}\sqrt{L})f\|_p\big)^q\Big)^{1/q}
\le c\delta^s\|f\|_{B^s_{pq}}.
\end{align*}
For the second inequality above we use H\"{o}lder's inequality if $q>1$
and the $q$-inequality if $0<q\le 1$;
for the last inequality we used the fact that the definition of the Besov space $B^s_{pq}$
is independent of the particular choice of the functions $\Phi_0$ and $\Phi$ satisfying \eqref{Phi0}-\eqref{Phi}
in its definition, hance $\bPhi$ produces an equivalent norm.

The case when $q=\infty$ is as easy; we omit it.
The proof of Proposition~\ref{prop:approx} is complete.
\qed

\end{document}